\documentclass[hidelinks,onefignum,onetabnum]{siamart251216}
\usepackage{algorithmic}

\usepackage{amsfonts}
\usepackage{amsmath}
\usepackage{amssymb}
\usepackage{mathtools}
\usepackage{booktabs}
\usepackage{enumitem}
\usepackage{array}
\usepackage{graphicx}

\newsiamremark{remark}{Remark}

\headers{The Projected Hessian Quantification Theorem}{Meiling Wang and Yong Xia}

\title{The Projected Hessian Quantification Theorem: An Exact Duality for Constrained Eigenvalues \thanks{\today.
\funding{This work was supported by National Natural Science Foundation of China (Grant No. 125B2016, 12631012).}}}
\author{Meiling Wang\thanks{School of System Science and Statistics, Beijing Wuzi University, Beijing 101149, People’s Republic of China.
  (\email{wangmeiling@bwu.edu.cn}).}
\and Yong Xia \thanks{Corresponding author. LMIB of the Ministry of Education, School of Mathematical Sciences, Beihang University, Beijing 100191, People’s Republic of China.
  (\email{yxia@buaa.edu.cn}).}}

\ifpdf
\hypersetup{
  pdftitle={The Projected Hessian Quantification Theorem: An Exact Duality for Constrained Eigenvalues},
  pdfauthor={Meiling Wang, and Yong Xia}
}
\fi

\begin{document}

\maketitle

\begin{abstract}
The classical Projected Hessian Lemma, originating from Finsler's theorem, characterizes definiteness over a constraint null space through quadratic penalties.
However, it does not quantify the corresponding constrained eigenvalues or the associated eigenvalue penalty path.
This work develops a quantitative penalty theory for constrained symmetric eigenvalue problems and establishes exact characterizations of the extremal eigenvalues of the reduced Hessian via full-space penalized eigenvalue problems.
Three proofs are provided based on orthogonal decomposition, Schur complement analysis, and semidefinite programming duality.
We characterize finite exact recovery along the extremal-eigenvalue penalty paths and, in the absence of finite recovery, establish asymptotic convergence with a first-order error expansion and an explicit leading coefficient.
The associated Hellmann--Feynman sensitivity relation leads to a strategy for predicting penalty parameters.
Based on these results, we develop a matrix-free Penalty--Split--Merge method for successive constrained extremal eigenpairs using penalty continuation, Split--Merge iterations, deflation, and projected certification.
Numerical experiments illustrate the predicted penalty regimes, evaluate projected certification, and assess computational performance on moderate-scale benchmarks and large-scale matrix-free test instances.
\end{abstract}



\section{Introduction}
\label{sec:introduction}

Constrained eigenvalue problems arise in optimization, stability
analysis, structural mechanics, and other applications involving
spectral information on a prescribed subspace.
Let
$H\in\mathbb{S}^n$ be a real symmetric matrix (not necessarily a
Hessian), and let $A\in\mathbb{R}^{m\times n}$ have full row rank
with $m<n$. If $Z\in\mathbb{R}^{n\times(n-m)}$ has orthonormal columns
spanning $\mathcal{N}(A)$,
then the smallest eigenvalue of the restriction of $H$ to the feasible subspace is
$
\lambda_{\min}(Z^{\mathrm T}HZ).
$
Although this reduced-space formulation is natural, explicitly
constructing $Z$ can be unattractive for large-scale sparse problems:
a null-space basis may be dense and can destroy the sparsity or
matrix-free structure of the original operators. This motivates
full-space formulations in which the constraint is imposed implicitly
through a penalty mechanism.

A classical foundation for this viewpoint is provided by Finsler's
theorem~\cite{Finsler37}. In the present setting, it yields the following
projected Hessian characterization:

\begin{lemma}[Projected Hessian Lemma, Finsler, 1937~\cite{Finsler37}]
\label{thm:Finsler}
Let $H\in\mathbb{S}^n$, and suppose that
$A\in\mathbb{R}^{m\times n}$ has full row rank with $m<n$.
Let $Z$ have orthonormal columns spanning $\mathcal{N}(A)$. Then
\[
Z^{\mathrm T}HZ\succ0
~~\Longleftrightarrow~~
\text{there exists }\bar\rho>0\text{ such that }
H+\rho A^{\mathrm T}A\succ0
\text{ for every }\rho>\bar\rho .
\]
\end{lemma}

Although the terminology "Projected Hessian" is commonly used in
optimization, the matrix $H$ in this paper is a general real symmetric
matrix and is not required to be a Hessian.
At the positive-semidefinite boundary, finite augmentation requires an
additional compatibility condition. Anstreicher and
Wright~\cite{AnstreicherWright00} established the following
characterization:

\begin{lemma}[Anstreicher--Wright, 2000~\cite{AnstreicherWright00}]
\label{thm:AnstreicherWright}
Let $H$, $A$, $Z$ be as in
Lemma~\ref{thm:Finsler}. Suppose that
$Z^{\mathrm T}HZ\succeq0$ is singular. Then there exists
$\bar\rho>0$ such that
$H+\rho A^{\mathrm T}A\succeq0$ for every
$\rho\geq\bar\rho$ if and only if
$
\mathcal{N}(Z^{\mathrm T}HZ)
=
\mathcal{N}(Z^{\mathrm T}H^2Z).
$
In this case, $H+\rho A^{\mathrm T}A$ is singular for all sufficiently
large $\rho$.
\end{lemma}

These results characterize definiteness and finite positive-semidefinite
augmentation. However, they focus on qualitative properties and do not
provide a quantitative description of the associated penalty path.
In particular, the relation between
$\lambda_{\min}(Z^{\mathrm T}HZ)$ and
$\lambda_{\min}(H+\rho A^{\mathrm T}A)$, the conditions for finite
attainment of the limiting eigenvalue, and the convergence rate in the
nonattainable case require further analysis.

The problem is connected with several strands of optimization and
numerical linear algebra. Finsler-type results are closely related
to the S-lemma, projection lemmas, and semidefinite representations of
quadratic inequalities
\cite{PolikTerlaky07,BoydElGhaouiFeronBalakrishnan94,
VandenbergheBoyd96}.
In equality-constrained quadratic optimization,
reduced-space and saddle-point formulations provide classical
alternatives to explicit null-space reduction
\cite{GouldHribarNocedal01,Benzi2005}. Quadratic penalty and augmented
formulations are standard tools in constrained optimization
\cite{Hestenes69,Powell69,NocedalWright06}.

A more direct connection is linearly constrained Rayleigh quotient
optimization. Zhou, Bai, and Li~\cite{ZhouBaiLi21} developed theoretical
characterizations together with Krylov projection methods for such
problems. Our focus is different: rather than directly reformulating the
constrained Rayleigh quotient problem, we study the spectral penalty path
generated by homogeneous linear constraints. This leads to an exact penalty characterization of the constrained
eigenvalue, a criterion for finite attainment, and quantitative
asymptotics when finite attainment fails.

For large-scale symmetric eigenvalue computation, Lanczos and LOBPCG
are established iterative methods
\cite{Parlett1998,Saad11,Knyazev2001}. Recent optimization-based and
preconditioned approaches include EPIC and Riemannian preconditioning
\cite{ShaoChenBai25,ShaoChen25}. The Split--Merge method
\cite{LiuSongXia2026} provides a difference-based framework for dominant
eigenvalue computation and supports matrix-free operator evaluations.

In this paper, we study the penalty path
$\lambda_{\min}(H+\rho A^{\mathrm T}A)$ and its exact relation with
the constrained eigenvalue $\lambda_{\min}(Z^{\mathrm T}HZ)$. Our main
contributions are as follows:

\begin{itemize}

\item We establish an exact penalty characterization of the constrained
minimum eigenvalue as the supremum of a full-space penalized
eigenvalue path; by monotonicity, this supremum equals its
large-penalty limit. Three complementary proofs based on orthogonal
decomposition, Schur complement analysis, and semidefinite
programming duality are given.
Maximum-eigenvalue results and extensions to stacked linear
constraints are obtained by the same arguments.

\item We characterize finite
exact recovery and, when it fails, derive a first-order asymptotic
expansion with an explicit leading coefficient.
This distinguishes finite exact recovery from asymptotic recovery in the nonattainable case and provides the sensitivity information used for
penalty prediction.

\item We develop a matrix-free Penalty--Split--Merge (PSM) framework for successive constrained extremal eigenpairs.
The method combines penalty continuation based on the Hellmann--Feynman sensitivity relation, Split--Merge inner solves,
deflation, and projected certification.
It accesses $H$, $A$, and $A^{\mathrm T}$ through matrix-vector products and avoids explicit null-space bases and penalty matrices.

\item Numerical experiments illustrate the predicted penalty regimes
and first-order behavior and evaluate projected certification. PSM is
tested on moderate-scale benchmarks and large-scale matrix-free
instances with attainable solutions.

\end{itemize}

The remainder of the paper is organized as follows.
Section~\ref{sec:main_result} establishes the Projected Hessian
Quantification Theorem, together with its maximum-eigenvalue and
stacked-constraint extensions.
Section~\ref{sec:attainability_convergence} characterizes finite
attainability and derives the first-order asymptotic behavior.
Section~\ref{sec:penalty_sm} develops the matrix-free
Penalty--Split--Merge framework.
Section~\ref{sec:numerical_experiments} presents the numerical
experiments, and Section~\ref{sec:conclusion} concludes the paper.

Throughout the paper, all matrices and vectors are real.
$\mathbb{S}^n$ denotes the space of $n\times n$ real symmetric
matrices. The null space and range of a matrix $A$ are denoted by
$\mathcal{N}(A)$ and $\mathcal{R}(A)$, respectively.
The notation $\|\cdot\|$ denotes the Euclidean norm for vectors and
the spectral norm for matrices. For a symmetric matrix $M$,
$\lambda_{\min}(M)$ and $\lambda_{\max}(M)$ denote its smallest and
largest eigenvalues, respectively. The symbols $M\succ0$ and
$M\succeq0$ denote positive definiteness and positive
semidefiniteness, respectively.

\section{Projected Hessian quantification theorem and three proofs}
\label{sec:main_result}

This section establishes the main theoretical result. The
smallest eigenvalue of the Hessian restricted to the nullspace of the
constraint matrix admits an exact characterization through the smallest
eigenvalue of a penalty-augmented matrix in the full space.

\begin{theorem}[Projected Hessian Quantification Theorem]
\label{thm:main_duality}
Let $H\in\mathbb{S}^n$, let $A\in\mathbb{R}^{m\times n}$ have full
row rank with $m<n$, and let
$Z\in\mathbb{R}^{n\times(n-m)}$ have orthonormal columns spanning
$\mathcal{N}(A)$. Then
\begin{equation}
\label{eq:main_equality}
\lambda_{\min}(Z^{\mathrm{T}}HZ)
=
\sup_{\rho\in\mathbb{R}}
\lambda_{\min}\bigl(H+\rho A^{\mathrm{T}}A\bigr).
\end{equation}
\end{theorem}

In what follows, we use the notation
\[
f(\rho):=\lambda_{\min}\bigl(H+\rho A^{\mathrm{T}}A\bigr),
~~
\lambda_*:=\lambda_{\min}(Z^{\mathrm{T}}HZ).
\]
Thus Theorem~\ref{thm:main_duality} states that
$\lambda_*=\sup\limits_{\rho\in\mathbb{R}}f(\rho)$.
Three proofs are presented using complementary approaches:
orthogonal decomposition with compactness arguments, Schur complement
analysis, and semidefinite programming duality.

\subsection{Preliminaries and block decomposition}
\label{subsec:preliminary_lemmas}

We first recall the Rayleigh--Ritz variational characterization, which
will be used repeatedly; and then introduce the block decomposition
associated with $\mathcal{N}(A)$ and $\mathcal{R}(A^{\mathrm{T}})$.

\begin{theorem}[Rayleigh--Ritz Theorem \cite{Parlett1998,Strang2016}]
\label{thm:Rayleigh_Ritz}
Let $M\in\mathbb{S}^n$. Then
\begin{align}
\lambda_{\min}(M)
=
\min_{\lVert x\rVert=1}x^{\mathrm{T}}Mx,
~~
\lambda_{\max}(M)
=
\max_{\lVert x\rVert=1}x^{\mathrm{T}}Mx.
\label{eq:RR_std_max}
\end{align}
\end{theorem}

We now separate the nullspace of $A$ from its orthogonal complement.

\begin{lemma}[Orthogonal Decomposition and Block Representation]
\label{lem:orthogonal_interlacing}
Let $A\in\mathbb{R}^{m\times n}$ have full row rank with $m<n$, and
let $H\in\mathbb{S}^n$. Let
$U\in\mathbb{R}^{n\times m}$ and
$Z\in\mathbb{R}^{n\times(n-m)}$ have orthonormal columns spanning
$\mathcal{R}(A^{\mathrm{T}})$ and $\mathcal{N}(A)$, respectively.
Define
\begin{equation}\label{QBEDC}
Q:=[U,Z],
~
\Sigma:=AU,
~
B:=U^{\mathrm{T}}HU,
~
E:=U^{\mathrm{T}}HZ,
~
D:=Z^{\mathrm{T}}HZ,
~
C:=\Sigma^{\mathrm{T}}\Sigma.
\end{equation}
Then the following statements hold.
\begin{enumerate}[leftmargin=2.8em]
\item[\textnormal{(i)}]
$Q$ is orthogonal, $\Sigma$ is nonsingular, and $C\succ0$.

\item[\textnormal{(ii)}]
Every $x\in\mathbb{R}^n$ has a unique decomposition
$
x=Uy+Zz,
~~
y\in\mathbb{R}^m,~~ z\in\mathbb{R}^{n-m}.
$
Moreover,
$
Ax=\Sigma y,
~~
\lVert x\rVert^2=\lVert y\rVert^2+\lVert z\rVert^2.
$

\item[\textnormal{(iii)}]
For $x=Uy+Zz$ and every $\rho\in\mathbb{R}$,
\begin{equation}
\label{eq:quadratic_decomposition}
x^{\mathrm{T}}
\bigl(H+\rho A^{\mathrm{T}}A\bigr)x
=
y^{\mathrm{T}}(B+\rho C)y
+
2y^{\mathrm{T}}Ez
+
z^{\mathrm{T}}Dz.
\end{equation}

\item[\textnormal{(iv)}]
The corresponding block representation is, for every
$\rho\in\mathbb{R}$,
\begin{equation}
\label{eq:block_penalty_matrix}
Q^{\mathrm{T}}
\bigl(H+\rho A^{\mathrm{T}}A\bigr)Q
=
\begin{bmatrix}
B+\rho C & E\\
E^{\mathrm{T}} & D
\end{bmatrix}.
\end{equation}

\item[\textnormal{(v)}]
The reduced-space variational identity is
\begin{equation}
\label{eq:nullspace_variational}
\lambda_{\min}(D)
=
\min_{\substack{x\in\mathcal{N}(A)\\ \lVert x\rVert=1}}
x^{\mathrm{T}}Hx
=
\lambda_*.
\end{equation}
\end{enumerate}
\end{lemma}

\begin{proof}
Since $A$ has full row rank,
$\mathcal{R}(A^{\mathrm T})=\mathcal{N}(A)^\perp$.
Therefore, $U$ and $Z$ span orthogonal complementary subspaces,
and $Q=[U,Z]$ is orthogonal.
If $\Sigma y=0$, then
$
Uy\in\mathcal{N}(A)\cap\mathcal{R}(A^{\mathrm{T}})=\{0\}.
$
Since $U$ has orthonormal columns, $y=0$. Hence $\Sigma$ is
nonsingular, and therefore $C=\Sigma^{\mathrm{T}}\Sigma\succ0$.
This proves part~\textnormal{(i)}.

Since $Q$ is orthogonal, every $x\in\mathbb{R}^n$ has the unique
representation
$
x=Uy+Zz,
~~
y=U^{\mathrm{T}}x,~~ z=Z^{\mathrm{T}}x.
$
Using $AZ=0$ and $AU=\Sigma$, we obtain $Ax=\Sigma y$.
Orthogonality also gives
$
\lVert x\rVert^2=\lVert y\rVert^2+\lVert z\rVert^2.
$
This proves part~\textnormal{(ii)}.

Expanding the quadratic form gives
\eqref{eq:quadratic_decomposition}.
$
Q^{\mathrm{T}}A^{\mathrm{T}}AQ
=
\begin{bmatrix}
C & 0\\
0 & 0
\end{bmatrix}
$
yields the block representation
\eqref{eq:block_penalty_matrix}.
Finally, every unit vector $x\in\mathcal{N}(A)$ can be written uniquely
as $x=Zz$ with $\lVert z\rVert=1$.
Therefore, part~\textnormal{(v)} follows by the Rayleigh--Ritz
characterization in Theorem~\ref{thm:Rayleigh_Ritz}.
\end{proof}

The next lemma records the basic properties of the penalty path.

\begin{lemma}[Basic Spectral Properties of the Penalty Function]
\label{lem:spectral_bounds}
Under the assumptions of Theorem~\ref{thm:main_duality}, the function
$f$ is globally Lipschitz continuous, concave, and nondecreasing on
$\mathbb{R}$. Moreover,
$
f(\rho)\leq\lambda_*
~~
\text{for every }\rho\in\mathbb{R}.
$
\end{lemma}

\begin{proof}
By Theorem~\ref{thm:Rayleigh_Ritz},
$
f(\rho)
=
\min\limits_{\lVert x\rVert=1}
\left(
x^{\mathrm{T}}Hx+\rho\lVert Ax\rVert^2
\right).
$
For each fixed unit vector $x$, $x^{\mathrm{T}}Hx+\rho\lVert Ax\rVert^2$ is affine and nondecreasing in $\rho$. Hence $f$ is concave and nondecreasing.
For any $\rho_1,\rho_2\in\mathbb{R}$, Weyl's perturbation inequality
gives
$
\lvert f(\rho_2)-f(\rho_1)\rvert
\leq
\lvert\rho_2-\rho_1\rvert
\lVert A^{\mathrm{T}}A\rVert.
$
Thus $f$ is globally Lipschitz continuous.
Finally,
\[
\begin{aligned}
f(\rho)
=
\min_{\lVert x\rVert=1}
x^{\mathrm{T}}\bigl(H+\rho A^{\mathrm{T}}A\bigr)x
\leq
\min_{\substack{x\in\mathcal{N}(A)\\ \lVert x\rVert=1}}
x^{\mathrm{T}}\bigl(H+\rho A^{\mathrm{T}}A\bigr)x
=
\min_{\substack{x\in\mathcal{N}(A)\\ \lVert x\rVert=1}}
x^{\mathrm{T}}Hx
=
\lambda_*.
\end{aligned}
\]
\end{proof}

\subsection{Proof via orthogonal decomposition}
\label{subsec:orthogonal_proof}

The first proof combines the orthogonal decomposition in
Lemma~\ref{lem:orthogonal_interlacing} with a compactness argument. The
penalty term forces minimizing vectors toward $\mathcal{N}(A)$ as
$\rho\to+\infty$. Since Lemma~\ref{lem:spectral_bounds} has already given
$f(\rho)\leq\lambda_*$, it remains to establish the matching lower
bound for the supremum.

\begin{proof}[Proof of Theorem~\ref{thm:main_duality} via Orthogonal Decomposition]
By Lemma~\ref{lem:spectral_bounds}, $f$ is nondecreasing and
$f(\rho)\leq\lambda_*$ for every $\rho\in\mathbb{R}$. Hence
$L:=\lim\limits_{\rho\to+\infty}f(\rho)$ exists and satisfies
$L\leq\lambda_*$. Suppose, for contradiction, that $L<\lambda_*$.
Choose $\epsilon>0$ such that
$L+2\epsilon<\lambda_*$, and select a sequence
$\{\rho_j\}$ satisfying $\rho_j\to+\infty$.
Then, for all sufficiently large $j$,
$f(\rho_j)\leq L+\epsilon\leq\lambda_*-\epsilon$.

By Theorem~\ref{thm:Rayleigh_Ritz}, choose a unit eigenvector $x_j$
associated with $f(\rho_j)$, so that
$f(\rho_j)=x_j^{\mathrm T}Hx_j+\rho_j\|Ax_j\|^2$.
Using Lemma~\ref{lem:orthogonal_interlacing}, write
$x_j=Uy_j+Zz_j$, where
$\|y_j\|^2+\|z_j\|^2=1$. Let
$\sigma:=\lambda_{\min}(C)>0$. Since $Ax_j=\Sigma y_j$,
$\|Ax_j\|^2=y_j^{\mathrm T}Cy_j\geq\sigma\|y_j\|^2$, while
$x_j^{\mathrm T}Hx_j\geq\lambda_{\min}(H)$. Therefore, for all
sufficiently large $j$,
\[
0\leq
\rho_j\sigma\|y_j\|^2
\leq
\lambda_*-\lambda_{\min}(H).
\]
Since $\rho_j\to+\infty$, it follows that $y_j\to0$ and hence
$\|z_j\|\to1$.

Since $\{z_j\}$ is bounded, there exists a subsequence, still denoted
by $\{z_j\}$, such that $z_j\to z_*$.
Since $\|z_j\|\to1$, we have
$\|z_*\|=1$. Along the same subsequence,
$x_j=Uy_j+Zz_j\to Zz_*$, and hence
\[
x_j^{\mathrm T}Hx_j
\longrightarrow
z_*^{\mathrm T}Dz_*
\geq
\lambda_{\min}(D)
=
\lambda_*.
\]
Since $\rho_j\to+\infty$, we have $\rho_j\geq0$ for all
sufficiently large $j$. Hence, by the Rayleigh representation,
\[
f(\rho_j)
=
x_j^{\mathrm T}Hx_j+\rho_j\|Ax_j\|^2
\geq
x_j^{\mathrm T}Hx_j .
\]
Along this subsequence,
\[
\liminf_{j\to\infty} f(\rho_j)\geq\lambda_*,
\]
which contradicts the bound
$f(\rho_j)\leq\lambda_*-\epsilon$ for sufficiently large $j$.
Therefore, $L=\lambda_*$.
Since
$f$ is nondecreasing,
$
\sup\limits_{\rho\in\mathbb{R}}f(\rho)=L=\lambda_*,
$
which proves \eqref{eq:main_equality}.
\end{proof}

\subsection{Proof via Schur complement}
\label{subsec:schur_proof}

The second proof uses the block representation in
Lemma~\ref{lem:orthogonal_interlacing} together with Schur complement. Rather than tracking minimizing vectors, it converts the
desired spectral lower bound into a positive-semidefiniteness condition
for sufficiently large $\rho$.
We recall the standard Schur complement and its positive-semidefinite
criterion; see \cite[Chap.~7]{HornJohnson12}.

\begin{lemma}[Schur Complement and Positive-Semidefinite Criterion]
\label{lem:schur_psd_criterion}
Let
$M=\begin{bmatrix}P&Y\\R&S\end{bmatrix}$, where
$P\in\mathbb{R}^{n\times n}$,
$Y\in\mathbb{R}^{n\times m}$,
$R\in\mathbb{R}^{m\times n}$, and
$S\in\mathbb{R}^{m\times m}$ is nonsingular.
Its Schur complement with respect to $S$ is
$M/S:=P-YS^{-1}R$.
If, in addition, $P\in\mathbb{S}^n$, $S\in\mathbb{S}^m$,
$R=Y^{\mathrm T}$, and $S\succ0$, then
\[
\begin{bmatrix}
P & Y\\
Y^{\mathrm T} & S
\end{bmatrix}
\succeq0
\quad\Longleftrightarrow\quad
P-YS^{-1}Y^{\mathrm T}\succeq0.
\]
\end{lemma}

\begin{proof}[Proof of Theorem~\ref{thm:main_duality} via Schur Complement]
With the notation of Lemma~\ref{lem:orthogonal_interlacing}, consider
the block representation \eqref{eq:block_penalty_matrix}. By
Lemma~\ref{lem:spectral_bounds},
$f(\rho)\leq\lambda_{\min}(D)=\lambda_*$ for every
$\rho\in\mathbb{R}$. It remains to establish the lower bound.

Fix $\gamma<\lambda_{\min}(D)$. Then
$D-\gamma I_{n-m}\succ0$, and the orthogonal transformation by $Q$ gives
\[
Q^{\mathrm T}
\bigl(H+\rho A^{\mathrm T}A-\gamma I_n\bigr)Q
=
\begin{bmatrix}
B-\gamma I_m+\rho C & E\\
E^{\mathrm T} & D-\gamma I_{n-m}
\end{bmatrix}.
\]
By Lemma~\ref{lem:schur_psd_criterion}, this matrix is positive
semidefinite if and only if
\begin{equation}\label{Krho}
B-\gamma I_m+\rho C
-
E(D-\gamma I_{n-m})^{-1}E^{\mathrm T}
\succeq0.
\end{equation}
Define
$K_\gamma:=B-\gamma I_m
-E(D-\gamma I_{n-m})^{-1}E^{\mathrm T}$.
The condition \eqref{Krho} becomes $K_\gamma+\rho C\succeq0$. Since $C\succ0$,
\[
\lambda_{\min}(K_\gamma+\rho C)
\geq
\lambda_{\min}(K_\gamma)
+\rho\lambda_{\min}(C),
\]
and hence $K_\gamma+\rho C\succeq0$ for all sufficiently large
$\rho>0$.

Because $Q$ is orthogonal,
$H+\rho A^{\mathrm T}A-\gamma I_n\succeq0$, therefore
$f(\rho)\geq\gamma$ for all sufficiently large $\rho$.
Thus
$\sup\limits_{\rho\in\mathbb{R}}f(\rho)\geq\gamma$.
Since this holds for every $\gamma<\lambda_*$, letting
$\gamma\uparrow\lambda_*$ gives
$\sup\limits_{\rho\in\mathbb{R}}f(\rho)\geq\lambda_*$.
Together with $f(\rho)\leq\lambda_*$ for every $\rho\in\mathbb{R}$,
we obtain
$
\sup\limits_{\rho\in\mathbb{R}}f(\rho)=\lambda_*,
$
which proves \eqref{eq:main_equality}.
\end{proof}

Augmented block formulations and their Schur-complement reductions are
standard in the saddle-point literature; see, e.g., \cite{Benzi2005}.
They yield the following feasibility representations.

\begin{corollary}[Schur-Complement Feasibility Representations]
\label{cor:schur_representation}
Under the assumptions of Theorem~\ref{thm:main_duality}, for
$\rho>0$, define
\[
\mathcal{M}_{+}(\rho)
:=
\begin{bmatrix}
H & A^{\mathrm T}\\
A & -\dfrac{1}{\rho}I_m
\end{bmatrix},
~~
\mathcal{M}_{-}(\rho)
:=
\begin{bmatrix}
H & -A^{\mathrm T}\\
-A & \dfrac{1}{\rho}I_m
\end{bmatrix}.
\]
Their Schur complements with respect to the lower-right blocks are
$H+\rho A^{\mathrm T}A$ and $H-\rho A^{\mathrm T}A$,
respectively. Moreover,
\begin{align}
\lambda_{\min}(Z^{\mathrm T}HZ)
&=
\sup\left\{
\gamma\in\mathbb{R}:
\exists\,\rho>0,\
H+\rho A^{\mathrm T}A\succeq\gamma I_n
\right\},
\label{eq:schur_min_feasibility}
\\
\lambda_{\max}(Z^{\mathrm T}HZ)
&=
\inf\left\{
\eta\in\mathbb{R}:
\exists\,\rho>0,\
H-\rho A^{\mathrm T}A\preceq\eta I_n
\right\}.
\label{eq:schur_max_feasibility}
\end{align}
\end{corollary}

\begin{proof}
Since $f$ is nondecreasing,
$
\sup\limits_{\rho\in\mathbb{R}} f(\rho)
=
\sup\limits_{\rho>0} f(\rho).
$
Hence Theorem~\ref{thm:main_duality} gives
\[
\lambda_{\min}(Z^{\mathrm T}HZ)
=
\sup_{\rho>0}
\lambda_{\min}\bigl(H+\rho A^{\mathrm T}A\bigr).
\]
For fixed $\rho>0$,
$\lambda_{\min}(H+\rho A^{\mathrm T}A)\geq\gamma \Longleftrightarrow H+\rho A^{\mathrm T}A\succeq\gamma I_n$, proving
\eqref{eq:schur_min_feasibility}.

Applying Theorem~\ref{thm:main_duality} to $-H$ and using the monotonicity similarly gives
\[
\lambda_{\max}(Z^{\mathrm T}HZ)
=
\inf_{\rho>0}
\lambda_{\max}\bigl(H-\rho A^{\mathrm T}A\bigr).
\]
For fixed $\rho>0$,
$\lambda_{\max}(H-\rho A^{\mathrm T}A)\leq\eta \Longleftrightarrow H-\rho A^{\mathrm T}A\preceq\eta I_n$, proving
\eqref{eq:schur_max_feasibility}.
\end{proof}

Thus $\mathcal{M}_{+}(\rho)$ and $\mathcal{M}_{-}(\rho)$ encode the
feasibility conditions through their Schur complements rather than
through their own spectra.

\subsection{Proof via semidefinite programming duality}
\label{subsec:optimization_proof}

The third proof uses semidefinite programming duality. Introducing
$X=xx^{\mathrm T}$ gives a rank-one semidefinite formulation; dropping
the rank constraint yields an SDP relaxation that is exact in optimal
value. The dual variable associated with
$\langle A^{\mathrm T}A,X\rangle=0$ is precisely the penalty parameter
$\rho$.

Consider
\begin{equation}
\label{eq:primal_problem_original}
p^*
=
\min\left\{
x^{\mathrm T}Hx:
Ax=0,\ \|x\|=1
\right\}.
\end{equation}
By Theorem~\ref{thm:Rayleigh_Ritz},
$p^*=\lambda_{\min}(Z^{\mathrm T}HZ)=\lambda_*$.
For matrices of the same size, let
$\langle M,N\rangle:=\operatorname{Tr}(M^{\mathrm T}N)$.
Under $X=xx^{\mathrm T}$, the objective and constraints in
\eqref{eq:primal_problem_original} become
$\langle H,X\rangle$,
$\langle A^{\mathrm T}A,X\rangle=0$, and
$\langle I_n,X\rangle=1$, together with
$X\succeq0$ and $\operatorname{rank}(X)=1$.
Dropping the rank constraint gives
\begin{equation}
\label{eq:primal_sdp}
\begin{array}{cl}
\displaystyle\min_{X\in\mathbb{S}^n}
&
\langle H,X\rangle
\\
\textnormal{s.t.}
&
\langle A^{\mathrm T}A,X\rangle=0,\\
&
\langle I_n,X\rangle=1,\\
&
X\succeq0.
\end{array}
\end{equation}
Let $p_{\rm sdp}^*$ denote the optimal value of \eqref{eq:primal_sdp}.

The relaxation is exact. Indeed, if $X$ is feasible, then
$AXA^{\mathrm T}\succeq0$ and
\[
\operatorname{Tr}(AXA^{\mathrm T})
=
\langle A^{\mathrm T}A,X\rangle
=
0,
\]
so $AX^{1/2}=0$. Hence
$\mathcal{R}(X)\subseteq\mathcal{N}(A)=\mathcal{R}(Z)$ and
$X=ZYZ^{\mathrm T}$ for some $Y\succeq0$ with
$\operatorname{Tr}(Y)=1$. Conversely, every such $Y$ produces a
feasible $X$. Therefore
\begin{equation}
\label{eq:exact_sdp_value}
p_{\rm sdp}^*
=
\min_{\substack{Y\succeq0\\ \operatorname{Tr}(Y)=1}}
\langle Z^{\mathrm T}HZ,Y\rangle
=
\lambda_*
=
p^*.
\end{equation}

Introduce unrestricted multipliers $\rho,\mu\in\mathbb{R}$ for the two
equality constraints in \eqref{eq:primal_sdp}. The Lagrangian is
\begin{equation}
\label{eq:lagrangian}
\mathcal{L}(X,\rho,\mu)
=
\left\langle
H+\rho A^{\mathrm T}A-\mu I_n,X
\right\rangle+\mu.
\end{equation}
The dual function is
\begin{equation}
\label{eq:dual_function}
g(\rho,\mu)
:=
\inf_{X\succeq0}\mathcal{L}(X,\rho,\mu)
=
\begin{cases}
\mu,
&
H+\rho A^{\mathrm T}A-\mu I_n\succeq0,
\\[1mm]
-\infty,
&
\textnormal{otherwise}.
\end{cases}
\end{equation}
Hence the dual problem is \begin{equation} \label{eq:dual_sdp} \begin{array}{cl} \displaystyle\max_{\rho,\mu\in\mathbb{R}} & \mu \\ \textnormal{s.t.} & H+\rho A^{\mathrm T}A\succeq\mu I_n. \end{array} \end{equation} For fixed $\rho$, the largest feasible $\mu$ is $f(\rho)=\lambda_{\min}(H+\rho A^{\mathrm T}A)$. Thus \begin{equation} \label{eq:simplified_dual} d^* = \sup_{\rho\in\mathbb{R}}f(\rho). \end{equation}

\begin{proof}[Proof of Theorem~\ref{thm:main_duality} via Duality]
The primal SDP \eqref{eq:primal_sdp} is feasible and has finite
optimal value. Its dual is strictly feasible: taking $\rho=0$ and any
$\mu<\lambda_{\min}(H)$ gives $H-\mu I_n\succ0$. Hence the dual Slater
condition yields strong duality, so $p_{\rm sdp}^*=d^*$. Combining
\eqref{eq:exact_sdp_value} and \eqref{eq:simplified_dual} gives
$
\lambda_*
=
\sup\limits_{\rho\in\mathbb{R}}f(\rho),
$
which proves \eqref{eq:main_equality}.
\end{proof}

The duality argument also clarifies the role of the penalty parameter.
In the SDP, $\rho$ is the unrestricted multiplier associated with
$\langle A^{\mathrm T}A,X\rangle=0$. For $X\succeq0$, this constraint
implies
$\mathcal{R}(X)\subseteq\mathcal{N}(A)$.
Thus the full-space penalty representation provides a dual
interpretation of the reduced-space eigenvalue problem.

\begin{remark}
\label{rem:monotonicity_consequences}
The three proofs emphasize complementary aspects of the same identity:
the orthogonal-decomposition proof is geometric, the Schur-complement
proof is matrix analytic, and the semidefinite-programming proof is
optimization theoretic.
By the monotonicity of
$f(\rho)=\lambda_{\min}(H+\rho A^{\mathrm T}A)$ and
Theorem~\ref{thm:main_duality},
\[
\sup_{\rho\in\mathbb{R}}f(\rho)
=
\sup_{\rho\geq0}f(\rho)
=
\lim_{\rho\to+\infty}f(\rho)
=
\lambda_*.
\]
Moreover, once the supremum is attained at some finite $\rho^*$, it is
attained for all $\rho\geq\rho^*$. This observation underlies the
attainable regime characterized in
Theorem~\ref{thm:attainability}.
\end{remark}

\subsection{Maximum eigenvalues and stacked constraints}
\label{subsec:extensions_boundaries}

The corresponding maximum-eigenvalue result follows by applying
Theorem~\ref{thm:main_duality} to $-H$.

\begin{corollary}[Maximum-Eigenvalue Version]
\label{cor:max_dual}
Under the assumptions of Theorem~\ref{thm:main_duality},
\[
\lambda_{\max}(Z^{\mathrm T}HZ)
=
\inf_{\rho\in\mathbb{R}}
\lambda_{\max}\bigl(H-\rho A^{\mathrm T}A\bigr).
\]
\end{corollary}

Multiple linear constraints are handled by stacking their constraint
matrices.

\begin{corollary}[Stacked Linear Constraints]
\label{cor:stacked_constraints}
Let $A_i\in\mathbb{R}^{m_i\times n}$, $i=1,\ldots,k$, and define
$\mathcal A:=[A_1^{\mathrm T},\ldots,A_k^{\mathrm T}]^{\mathrm T}$.
Assume that $\mathcal A$ has full row rank and fewer than $n$ rows, and
let $Z$ have orthonormal columns spanning
$\mathcal N(\mathcal A)=\bigcap_{i=1}^k\mathcal N(A_i)$. Then
\[
\begin{aligned}
\lambda_{\min}(Z^{\mathrm T}HZ)
&=
\sup_{\rho\in\mathbb{R}}
\lambda_{\min}\left(
H+\rho\sum_{i=1}^k A_i^{\mathrm T}A_i
\right),\\
\lambda_{\max}(Z^{\mathrm T}HZ)
&=
\inf_{\rho\in\mathbb{R}}
\lambda_{\max}\left(
H-\rho\sum_{i=1}^k A_i^{\mathrm T}A_i
\right).
\end{aligned}
\]
\end{corollary}

\begin{remark}
If $\mathcal A$ has redundant rows, it may be replaced by any
full-row-rank matrix with the same row space. This preserves its null
space and hence the constrained eigenvalues.
\end{remark}

\section{Finite attainability and first-order convergence rate}
\label{sec:attainability_convergence}

The preceding section characterizes the constrained minimum eigenvalue through the full-space penalty path $f$.
We now analyze finite attainability and the asymptotic convergence rate of this penalty path.
The corresponding results for the maximum-eigenvalue penalty path follow immediately by applying the same arguments to $-H$.

\subsection{Finite attainability}
\label{subsec:attainability}

At $\lambda_*$, the lower-right block
$D-\lambda_*I_{n-m}$ in the Schur complement argument is positive
semidefinite and may be singular. We therefore use the following
standard extension of Lemma~\ref{lem:schur_psd_criterion}.

\begin{lemma}[Generalized Positive-Semidefinite Schur Complement Criterion]
\label{lem:generalized_schur}
Let $P\in\mathbb{S}^n$, $Y\in\mathbb{R}^{n\times m}$, and
$S\in\mathbb{S}^m$ with $S\succeq0$. Then
\[
\begin{bmatrix}
P & Y\\
Y^{\mathrm T} & S
\end{bmatrix}
\succeq0
~~\Longleftrightarrow~~
(I_m-SS^\dagger)Y^{\mathrm T}=0
~~\text{and}~~
P-YS^\dagger Y^{\mathrm T}\succeq0,
\]
where $S^\dagger$ is the Moore--Penrose inverse of $S$.
\end{lemma}

\begin{proof}
Suppose that the block matrix is positive semidefinite. For
$z\in\mathcal{N}(S)$, its quadratic form at
$[u^{\mathrm T},tz^{\mathrm T}]^{\mathrm T}$ is nonnegative for every
$u\in\mathbb{R}^n$ and $t\in\mathbb{R}$. Hence $Yz=0$, so
$\mathcal{R}(Y^{\mathrm T})\subseteq\mathcal{R}(S)$, equivalently
$
(I_m-SS^\dagger)Y^{\mathrm T}=0.
$
Under this range condition, $Y=YS^\dagger S$, and
\[
\begin{bmatrix}
P & Y\\
Y^{\mathrm T} & S
\end{bmatrix}
=
\begin{bmatrix}
I_n & YS^\dagger\\
0 & I_m
\end{bmatrix}
\begin{bmatrix}
P-YS^\dagger Y^{\mathrm T} & 0\\
0 & S
\end{bmatrix}
\begin{bmatrix}
I_n & 0\\
S^\dagger Y^{\mathrm T} & I_m
\end{bmatrix}.
\]
The two outer factors are nonsingular and are transposes of each other.
Therefore, by congruence, the block matrix is positive semidefinite if
and only if
$P-YS^\dagger Y^{\mathrm T}\succeq0$ and $S\succeq0$.
Since $S\succeq0$ is assumed, the stated equivalence follows.
\end{proof}
%

\begin{theorem}[Finite-Attainability Characterization]
\label{thm:attainability}
Under the assumptions of Theorem~\ref{thm:main_duality}, set
\[
K:=H-\lambda_*I_n,
\quad
\mathcal{E}_*
:=
\mathcal{N}(Z^{\mathrm T}HZ-\lambda_*I_{n-m}),
\]
where $\mathcal{E}_*$ is the eigenspace of
$Z^{\mathrm T}HZ$ associated with $\lambda_*$.
Then the following statements are equivalent.

\begin{enumerate}[leftmargin=2.8em]
\item[\textnormal{(i)}]
There exists a finite $\rho^*\in\mathbb{R}$ such that
$f(\rho^*)=\lambda_*$.

\item[\textnormal{(ii)}]
Every constrained minimizing direction is a full-space eigenvector of
$H$ associated with $\lambda_*$:
\begin{equation}
\label{eq:attain_cond}
KZw=0
~~\text{for every }w\in\mathcal{E}_*.
\end{equation}

\item[\textnormal{(iii)}]
The compressed kernels satisfy
\begin{equation}
\label{eq:attain_cond_kernel}
\mathcal{N}(Z^{\mathrm T}KZ)
=
\mathcal{N}(Z^{\mathrm T}K^2Z).
\end{equation}
\end{enumerate}

If these conditions hold, there exists $\bar\rho\geq0$ such that
$f(\rho)=\lambda_*$ for every $\rho\geq\bar\rho$. If they fail, then
$f(\rho)<\lambda_*$ for every finite $\rho$, $f$ is strictly increasing
on $\mathbb{R}$, and
$f(\rho)\uparrow\lambda_*$ as $\rho\to+\infty$.
\end{theorem}

\begin{proof}
For every $\rho\in\mathbb{R}$, Lemma~\ref{lem:spectral_bounds} gives
$
\lambda_{\min}\bigl(K+\rho A^{\mathrm T}A\bigr)
=
f(\rho)-\lambda_*
\leq0.
$
Hence
\begin{equation}
\label{eq:attainment_psd_equivalence}
f(\rho)=\lambda_*
~~\Longleftrightarrow~~
K+\rho A^{\mathrm T}A\succeq0.
\end{equation}

To prove \textnormal{(i)}$\Rightarrow$\textnormal{(ii)}, suppose
$f(\rho^*)=\lambda_*$, let $w\in\mathcal{E}_*$, and set $v:=Zw$.
Then
\[
Av=0,~~v^{\mathrm T}Kv
=
w^{\mathrm T}
\bigl(Z^{\mathrm T}HZ-\lambda_*I_{n-m}\bigr)w
=
0.
\]
By \eqref{eq:attainment_psd_equivalence},
$K+\rho^*A^{\mathrm T}A\succeq0$, and therefore
\[
v^{\mathrm T}
\bigl(K+\rho^*A^{\mathrm T}A\bigr)v
=
0
\quad\Longrightarrow\quad
\bigl(K+\rho^*A^{\mathrm T}A\bigr)v=0.
\]
Since $Av=0$, we have $A^{\mathrm T}Av=0$, and hence $Kv=0$,
proving \textnormal{(ii)}.

For the converse, assume \textnormal{(ii)} and retain the notation
in \eqref{QBEDC} from Lemma~\ref{lem:orthogonal_interlacing}. Set
$B_*:=B-\lambda_*I_m$ and
$D_*:=D-\lambda_*I_{n-m}\succeq0$. Then
\[
Q^{\mathrm T}(K+\rho A^{\mathrm T}A)Q
=
\begin{bmatrix}
B_*+\rho C & E\\
E^{\mathrm T} & D_*
\end{bmatrix}.
\]
If $w\in\mathcal{N}(D_*)=\mathcal{E}_*$, then
$
Q^{\mathrm T}KZw
=
\begin{bmatrix}
Ew\\
D_*w
\end{bmatrix},
$
so $KZw=0$ is equivalent to $Ew=0$. Hence
\eqref{eq:attain_cond} is equivalent to
$\mathcal{N}(D_*)\subseteq\mathcal{N}(E)$, or equivalently,
$
\mathcal{R}(E^{\mathrm T})\subseteq\mathcal{R}(D_*).
$
Under this range condition, Lemma~\ref{lem:generalized_schur} gives
\[
Q^{\mathrm T}(K+\rho A^{\mathrm T}A)Q\succeq0
~~\Longleftrightarrow~~
B_*+\rho C-ED_*^\dagger E^{\mathrm T}\succeq0.
\]
Since $C\succ0$,
$
\lambda_{\min}
\bigl(B_*-ED_*^\dagger E^{\mathrm T}+\rho C\bigr)
\geq
\lambda_{\min}
\bigl(B_*-ED_*^\dagger E^{\mathrm T}\bigr)
+
\rho\lambda_{\min}(C),
$
so there exists $\bar\rho\geq0$ such that the right-hand side is
positive semidefinite for every $\rho\geq\bar\rho$.
Since $Q$ is orthogonal, this implies
$K+\rho A^{\mathrm T}A\succeq0$, and
\eqref{eq:attainment_psd_equivalence} therefore yields
$f(\rho)=\lambda_*$ for every $\rho\geq\bar\rho$.
This proves \textnormal{(ii)}$\Rightarrow$\textnormal{(i)}.

To prove \textnormal{(ii)}$\Leftrightarrow$\textnormal{(iii)}, note that
$Z^{\mathrm T}KZ=D_*\succeq0$ and
$Z^{\mathrm T}K^2Z=(KZ)^{\mathrm T}(KZ)$. Hence
\[
\mathcal{N}(Z^{\mathrm T}K^2Z)
=
\{w:KZw=0\}
\subseteq
\mathcal{N}(Z^{\mathrm T}KZ)
=
\mathcal{E}_*.
\]
Equality holds if and only if
$KZw=0$ for every $w\in\mathcal{E}_*$, which is precisely
\eqref{eq:attain_cond}. Thus
\textnormal{(ii)} and \textnormal{(iii)} are equivalent.

Suppose now that \textnormal{(i)}--\textnormal{(iii)} fail. Then, by
\eqref{eq:attainment_psd_equivalence},
$f(\rho)<\lambda_*$ for every finite $\rho$. Since $f$ is
nondecreasing and Theorem~\ref{thm:main_duality} gives
$\sup\limits_{\rho\in\mathbb{R}}f(\rho)=\lambda_*$, we have
$f(\rho)\uparrow\lambda_*$ as $\rho\to+\infty$.
It remains to prove strict monotonicity. Suppose
$f(\rho_1)=f(\rho_2)=c$ for some $\rho_1<\rho_2$.
Since $f$ is nondecreasing, it is constant on
$[\rho_1,\rho_2]$. Choose $\rho_0\in(\rho_1,\rho_2)$ and a unit
eigenvector $x$ of $H+\rho_0A^{\mathrm T}A$ associated with $c$.
Then
\[
c
\leq
x^{\mathrm T}(H+\rho_1A^{\mathrm T}A)x
=
c-(\rho_0-\rho_1)\|Ax\|^2
\leq c.
\]
Hence $Ax=0$. Therefore
$c=x^{\mathrm T}Hx\geq\lambda_*$, contradicting
$c=f(\rho_0)<\lambda_*$. Thus $f$ is strictly increasing on
$\mathbb{R}$.
\end{proof}

\begin{remark}
\label{rem:attainability_geometry}
Finite attainability requires every constrained minimizing direction
to be a full-space eigenvector of $H$ associated with $\lambda_*$.
Hence, if $\lambda_*$ is a multiple eigenvalue of
$Z^{\mathrm T}HZ$, the condition must hold for every direction in
$\mathcal{E}_*$.
\end{remark}

The semidefinite formulation in
Subsection~\ref{subsec:optimization_proof} gives an equivalent dual
interpretation.

\begin{corollary}[Dual Attainment and Complementary Slackness]
\label{cor:dual_attainment}
Define
\[
S(\rho):=H+\rho A^{\mathrm T}A-\lambda_*I_n.
\]
A finite $\rho^*\in\mathbb{R}$ yields the dual optimal pair
$(\rho^*,\lambda_*)$ in \eqref{eq:dual_sdp} if and only if
$S(\rho^*)\succeq0$. In this case, every unit constrained minimizer
$x_*$ satisfies
$S(\rho^*)x_*=0$ and $Hx_*=\lambda_*x_*$.
Conversely, if every constrained minimizer is an eigenvector of $H$
associated with $\lambda_*$, then the dual optimum is attained at a
finite $\rho^*$.
\end{corollary}

\begin{proof}
By the strong-duality argument in
Subsection~\ref{subsec:optimization_proof}, the dual optimal value is
$\lambda_*$. Hence $(\rho^*,\lambda_*)$ is optimal exactly when
$S(\rho^*)\succeq0$. For a constrained minimizer $x_*$,
$Ax_*=0$ and $x_*^{\mathrm T}Hx_*=\lambda_*$, so
$x_*^{\mathrm T}S(\rho^*)x_*=0$. Since $S(\rho^*)\succeq0$, this
implies $S(\rho^*)x_*=0$, and hence
$Hx_*=\lambda_*x_*$. The converse follows directly from
Theorem~\ref{thm:attainability}.
\end{proof}

Thus finite spectral attainment and finite dual attainment describe the
same phenomenon: every constrained minimizing direction must be a
full-space eigenvector associated with $\lambda_*$.
Theorem~\ref{thm:main_duality},
Theorem~\ref{thm:attainability}, and the monotonicity of $f$ yield the
following definiteness consequences.

\begin{corollary}[Definiteness Consequences]
\label{cor:definiteness_consequences}
Under the assumptions of Theorem~\ref{thm:main_duality}, the following
reformulations hold.

\noindent
\textbf{Reformulation of Lemma~\ref{thm:Finsler}.}
The following statements are equivalent:
\begin{enumerate}[label=(\alph*),leftmargin=2.2em]
\item $Z^{\mathrm T}HZ\succ0$;
\item $\sup\limits_{\rho\in\mathbb{R}}f(\rho)>0$;
\item there exists $\rho_0\geq0$ such that
$H+\rho A^{\mathrm T}A\succ0$ for every $\rho\geq\rho_0$.
\end{enumerate}

\noindent
\textbf{Reformulation of Lemma~\ref{thm:AnstreicherWright}.}
If $\lambda_*=0$, the following statements are equivalent:
\begin{enumerate}[label=(\alph*),leftmargin=2.2em]
\item there exists $\rho_0\geq0$ such that
$H+\rho A^{\mathrm T}A\succeq0$ for every $\rho\geq\rho_0$;
\item there exists a finite $\rho_0\geq0$ such that $f(\rho_0)=0$;
\item $HZw=0$ for every
$w\in\mathcal{N}(Z^{\mathrm T}HZ)$;
\item
$\mathcal{N}(Z^{\mathrm T}HZ)
=
\mathcal{N}(Z^{\mathrm T}H^2Z)$.
\end{enumerate}
If these conditions fail, then $f(\rho)<0$ for every finite $\rho$,
$f$ is strictly increasing on $\mathbb{R}$, and
$f(\rho)\uparrow0$ as $\rho\to+\infty$.
\end{corollary}

The two parts recover Lemmas~\ref{thm:Finsler} and
\ref{thm:AnstreicherWright}, respectively. Thus the penalty path has
two distinct regimes: finite exact recovery and asymptotic recovery
from below. Next, we quantify the latter.

\subsection{First-order convergence rate}
\label{subsec:convergence_rate}
The next theorem identifies the leading-order asymptotic error through
the coupling, induced by $H$, between the constrained minimizing
eigenspace and $\mathcal{R}(A^{\mathrm T})$. The coefficient is
positive in the nonattainable case and vanishes under finite
attainability.

\begin{theorem}[First-Order Convergence Theorem]
\label{thm:main_convergence}
Under the assumptions of Theorem~\ref{thm:main_duality}, let
$r:=\dim(\mathcal{E}_*)$, and let
$W\in\mathbb{R}^{(n-m)\times r}$ have orthonormal columns spanning
$\mathcal{E}_*$. Then, as $\rho\to+\infty$,
\begin{equation}
\label{eq:main_convergence}
f(\rho)
=
\lambda_*-\frac{c}{\rho}
+
O\left(\frac{1}{\rho^2}\right),
\end{equation}
\begin{align}
c
&=
\lambda_{\max}
\left(
W^{\mathrm T}Z^{\mathrm T}H
A^{\mathrm T}(AA^{\mathrm T})^{-2}
AHZW
\right)
=
\max_{\substack{w\in\mathcal{E}_*\\ \|w\|=1}}
\left\|
(AA^{\mathrm T})^{-1}AHZw
\right\|^2.
\label{eq:c_constant_variational}
\end{align}
The coefficient $c$ is nonnegative, and $c=0$ if and only if the
finite-attainability conditions of Theorem~\ref{thm:attainability}
hold. Hence, in the nonattainable case, $c>0$ and
\[
\lambda_*-f(\rho)
=
\frac{c}{\rho}
+
O\left(\frac{1}{\rho^2}\right).
\]

If $\lambda_*$ is simple, let $v_*$ be a corresponding unit eigenvector
of $Z^{\mathrm T}HZ$ and set $x_*:=Zv_*$. Then
\begin{equation}
\label{eq:c_constant_simple}
c
=
(Hx_*)^{\mathrm T}
A^{\mathrm T}(AA^{\mathrm T})^{-2}A(Hx_*)
=
\left\|
(AA^{\mathrm T})^{-1}A(Hx_*)
\right\|^2.
\end{equation}
\end{theorem}

\begin{proof}
Retain the notation in \eqref{QBEDC} from
Lemma~\ref{lem:orthogonal_interlacing}. Set
$t:=1/\rho$ and $\lambda_t:=f(1/t)$. Since $f$ is nondecreasing and
Theorem~\ref{thm:main_duality} gives
$\sup\limits_{\rho\in\mathbb{R}}f(\rho)=\lambda_*$, we have
$\lambda_t\to\lambda_*$ as $t\downarrow0$; in particular,
$\lambda_t$ is bounded for sufficiently small $t>0$.

Let $[y_t^{\mathrm T},z_t^{\mathrm T}]^{\mathrm T}$ be a unit
eigenvector of the block matrix in
\eqref{eq:block_penalty_matrix}, with $\rho=t^{-1}$, associated with
$\lambda_t$. The corresponding block eigenvalue equations are
\begin{align}
\bigl(B+t^{-1}C-\lambda_tI_m\bigr)y_t+Ez_t&=0,
\label{eq:block_eig_first}\\
E^{\mathrm T}y_t+(D-\lambda_tI_{n-m})z_t&=0.
\label{eq:block_eig_second}
\end{align}

Since $C\succ0$,
$F_t:=B+t^{-1}C-\lambda_tI_m$ is nonsingular for sufficiently small
$t$, with $F_t^{-1}=tC^{-1}+O(t^2)$. Hence
$y_t=-F_t^{-1}Ez_t$, and substitution into
\eqref{eq:block_eig_second} gives
\[
(D-tG+R_t)z_t=\lambda_tz_t,
\quad
G:=E^{\mathrm T}C^{-1}E,
\quad
\|R_t\|=O(t^2),
\]
where $R_t:=tG-E^{\mathrm T}F_t^{-1}E$.
Since $F_t$ is nonsingular, $z_t\neq0$. Thus Weyl's inequality gives
$\lambda_t\geq\lambda_{\min}(D-tG)-\|R_t\|$. Since $\lambda_*$ is
the smallest eigenvalue of $D$ with eigenspace $\mathcal R(W)$, the
standard first-order perturbation formula for the corresponding
symmetric eigenvalue cluster yields
\[
\lambda_t
\geq
\lambda_*
-
t\lambda_{\max}(W^{\mathrm T}GW)
+
O(t^2).
\]

For the matching upper bound, choose a unit $q\in\mathbb{R}^r$ such
that
$q^{\mathrm T}W^{\mathrm T}GWq
=\lambda_{\max}(W^{\mathrm T}GW)$, and set
$z:=Wq$ and $y:=-tC^{-1}Ez$. Then
$Dz=\lambda_*z$ and $\|y\|^2+\|z\|^2=1+O(t^2)$.
By Theorem~\ref{thm:Rayleigh_Ritz}, $\lambda_t$ is bounded above by
the Rayleigh quotient of $[y^{\mathrm T},z^{\mathrm T}]^{\mathrm T}$.
A direct expansion therefore gives
\[
\lambda_t
\leq
\lambda_*
-
t\lambda_{\max}(W^{\mathrm T}GW)
+
O(t^2).
\]
Combining the two bounds and using $t=1/\rho$ gives
\[
f(\rho)
=
\lambda_*
-
\frac{\lambda_{\max}(W^{\mathrm T}GW)}{\rho}
+
O(\rho^{-2}).
\]

It remains to express the coefficient in terms of the original
operators. Since $A=\Sigma U^{\mathrm T}$,
$C=\Sigma^{\mathrm T}\Sigma$, and $E=U^{\mathrm T}HZ$,
\[
G
=E^{\mathrm T}C^{-1}E
=Z^{\mathrm T}HU(\Sigma^{\mathrm T}\Sigma)^{-1}U^{\mathrm T}HZ
=Z^{\mathrm T}HA^{\mathrm T}(AA^{\mathrm T})^{-2}AHZ,
\]
\[
W^{\mathrm T}GW
=
\left((AA^{\mathrm T})^{-1}AHZW\right)^{\mathrm T}
\left((AA^{\mathrm T})^{-1}AHZW\right).
\]
Since the columns of $W$ form an orthonormal basis of
$\mathcal E_*$, Theorem \ref{thm:Rayleigh_Ritz} gives
\eqref{eq:c_constant_variational}.

To characterize when $c=0$, use
$G=E^{\mathrm T}C^{-1}E$ to obtain
\[
W^{\mathrm T}GW
=
(C^{-1/2}EW)^{\mathrm T}(C^{-1/2}EW)
\succeq0.
\]
Thus $c\geq0$. Since $C\succ0$, we have
$c=0$ if and only if $EW=0.$
Because the columns of $W$ span $\mathcal E_*$ and
$
Q^{\mathrm T}(H-\lambda_*I_n)Zw
=
\begin{bmatrix}
Ew\\
0
\end{bmatrix}
,$
we have
$c=0$ if and only if
$(H-\lambda_*I_n)Zw=0$ for every $w\in\mathcal E_*$.
By Theorem~\ref{thm:attainability}, this is precisely the
finite-attainability condition. Thus $c>0$ in the nonattainable case.
If $\lambda_*$ is simple, then $r=1$ and we may take $W=v_*$.
With $x_*=Zv_*$, \eqref{eq:c_constant_variational} gives
\eqref{eq:c_constant_simple}.
\end{proof}

The variational characterization of $c$ measures the strongest
first-order coupling, through $H$, between the constrained minimizing
eigenspace and the constraint-normal space
$\mathcal R(A^{\mathrm T})$. It also shows that $c$ is independent of
the particular orthonormal bases used to represent these subspaces.
Together with Theorem~\ref{thm:attainability}, this yields the following
dichotomy.

\begin{corollary}[Finite-versus-Asymptotic Recovery]
\label{cor:convergence_patterns}
Under the assumptions of Theorem~\ref{thm:main_duality}, exactly one
of the following alternatives holds, with $c$ as in
Theorem~\ref{thm:main_convergence}.

\begin{enumerate}[leftmargin=2.8em]
\item[\textnormal{(i)}]
\textbf{Finite exact recovery.}
The equivalent conditions in Theorem~\ref{thm:attainability} hold,
or equivalently $c=0$. Then there exists $\bar\rho\geq0$ such that
$f(\rho)=\lambda_*$ for every $\rho\geq\bar\rho$.

\item[\textnormal{(ii)}]
\textbf{Strict asymptotic recovery.}
The equivalent conditions in Theorem~\ref{thm:attainability} fail,
or equivalently $c>0$. Then $f(\rho)<\lambda_*$ for every finite
$\rho$, $f$ is strictly increasing on $\mathbb{R}$, and, as
$\rho\to+\infty$,
\[
f(\rho)\uparrow\lambda_*,
\quad
\lambda_*-f(\rho)
=
\frac{c}{\rho}+O(\rho^{-2}).
\]
\end{enumerate}

In particular, if $\lambda_*=0$, either $f(\rho)=0$ for all
sufficiently large $\rho$, or $f(\rho)<0$ for every finite $\rho$ and
$f(\rho)=-c/\rho+O(\rho^{-2})$ with $c>0$.
\end{corollary}

The first-order expansion also yields a two-level extrapolation that
cancels the leading $1/\rho$ error.

\begin{corollary}[Two-Level Extrapolation]
\label{cor:richardson_extrapolation}
Under the assumptions of Theorem~\ref{thm:main_duality},
\[
2f(2\rho)-f(\rho)
=
\lambda_*+O(\rho^{-2}),
\quad \rho\to+\infty.
\]
\end{corollary}

\section{A matrix-free Penalty--Split--Merge framework}
\label{sec:penalty_sm}

The penalty characterizations above reduce the constrained extremal
eigenvalue problem to dominant eigenvalue computations along a penalty
path. We therefore combine penalty continuation with the
Split--Merge iteration~\cite{LiuSongXia2026}. The continuation is guided
by a safeguarded Hellmann--Feynman (HF)/KKT predictor, whereas final
acceptance is determined independently by the penalized eigen-residual
and projected feasibility, KKT, and outer-change tests.

\subsection{Penalty formulation and matrix-free Split--Merge inner solver}
\label{subsec:sm_inner}
Let $X\in\mathbb{R}^{n\times s}$ contain previously computed feasible
orthonormal eigenvectors. We assume $X^{\mathrm T}X=I_s$, $AX=0$, and
$0\leq s<n-m$; terms involving $X$ are absent when $s=0$. The next
eigenvector is sought in
$\mathcal S_X:=\mathcal N(A)\cap\operatorname{span}(X)^\perp$.
Define $\widehat A_X:=[A^{\mathrm T},X]^{\mathrm T}$. Then
$\mathcal N(\widehat A_X)=\mathcal S_X$ and
$\widehat A_X^{\mathrm T}\widehat A_X=A^{\mathrm T}A+XX^{\mathrm T}$.
Moreover, $AX=0$ makes the row spaces of $A$ and $X^{\mathrm T}$
orthogonal, so $\widehat A_X$ has full row rank $m+s<n$.
Since $A^{\mathrm T}A+XX^{\mathrm T}\succeq0$, the map
$\rho\mapsto\lambda_{\max}(M_\rho^{(X)})$ is nonincreasing.
Corollary~\ref{cor:max_dual} gives
\begin{equation}
\label{eq:penalty_characterization_sm}
\lambda_X^c
:=
\max_{\substack{u\in\mathcal S_X\\ \|u\|=1}}u^{\mathrm T}Hu
=
\inf_{\rho\geq0}\lambda_{\max}\!\left(M_\rho^{(X)}\right),
\end{equation}
where
$M_\rho^{(X)}=H-\rho(A^{\mathrm T}A+XX^{\mathrm T})$ is applied as
$M_\rho^{(X)}v=Hv-\rho A^{\mathrm T}(Av)-\rho X(X^{\mathrm T}v)$.
Thus neither $A^{\mathrm T}A$, $XX^{\mathrm T}$, nor a basis of
$\mathcal S_X$ is formed.

To apply Split--Merge to a positive semidefinite operator, let
$\underline\lambda_H\leq\lambda_{\min}(H)$,
$\overline\lambda_H\geq\lambda_{\max}(H)$, and
$\overline\nu_A\geq\|A\|^2$, and set
\begin{equation}
\label{eq:shift_scale}
\nu_X=
\begin{cases}
\overline\nu_A,&s=0,\\
\overline\nu_A+1,&s>0,
\end{cases}
~~
\eta_\rho=\max\{\rho\nu_X-\underline\lambda_H,0\},
~~
\mathcal K_\rho^{(X)}
=
\frac{M_\rho^{(X)}+\eta_\rho I}{\sigma_\rho},
\end{equation}
where
$\sigma_\rho=\max\{1,\overline\lambda_H+\eta_\rho\}$.
Since
$A^{\mathrm T}A+XX^{\mathrm T}\preceq\nu_XI$,
the spectrum of $\mathcal K_\rho^{(X)}$ lies in $[0,1]$; the shift and
positive scaling preserve eigenvectors and their ordering.
Equation~\eqref{eq:shift_scale} is a conservative valid choice: any
verified bound
$\nu_X\geq\|A^{\mathrm T}A+XX^{\mathrm T}\|$ may replace it without
changing $M_\rho^{(X)}$ or its eigenvectors.

For a normalized inner iterate $u$, let
$\mu=u^{\mathrm T}\mathcal K_\rho^{(X)}u$ and
$\theta^{\rm pen}=\sigma_\rho\mu-\eta_\rho
=u^{\mathrm T}M_\rho^{(X)}u$. Inner convergence is certified in the
original penalized operator by
\begin{equation}
\label{eq:inner_residual}
r_M
:=
\frac{\|M_\rho^{(X)}u-\theta^{\rm pen}u\|}
{\max\{1,|\theta^{\rm pen}|\}}
=
\frac{\sigma_\rho\|\mathcal K_\rho^{(X)}u-\mu u\|}
{\max\{1,|\theta^{\rm pen}|\}}
\leq\varepsilon_{\rm in}.
\end{equation}
The preceding converged vector is used as a warm start. The native
Split--Merge update is the two-product iteration of
\cite{LiuSongXia2026}; we summarize only the safeguards added here.
Within a fixed penalty level, let
$\lambda_{\mathcal K}(v)=v^{\mathrm T}\mathcal K_\rho^{(X)}v/\|v\|^2$
and retain the accepted iterate $u_{\rm best}$ with the largest observed
value $\lambda_{\rm best}$. A candidate $v$ is rejected if
$\lambda_{\mathcal K}(v)<
\lambda_{\rm best}-\delta_{\rm br}\max\{1,|\lambda_{\rm best}|\}$.
The residual is checked at a fixed interval; a check is regarded as
progress when it is at most $\tau_{\rm prog}$ times the best previously
checked residual, and $m_{\rm stag}$ consecutive unsuccessful checks
trigger a short Krylov--Ritz rescue. The threshold
$\varepsilon_\sigma$ guards against numerically unsafe native
Split--Merge coefficients; when this guard is activated, a normalized
power step is used. These safeguards never replace the residual
certificate \eqref{eq:inner_residual}. For the runtime comparisons in
Section~\ref{sec:numerical_experiments}, the rescue is realized by the
verified short-recurrence Lanczos implementation described there.

\subsection{Hellmann--Feynman sensitivity and penalty prediction}
\label{subsec:hf_prediction}

Define
\[
\phi_X(\rho):=\lambda_{\max}(M_\rho^{(X)}),
\quad
\mathcal V_X(u):=\|Au\|^2+\|X^{\mathrm T}u\|^2.
\]
By Theorem~\ref{thm:Rayleigh_Ritz},
\[
\phi_X(\rho)
=
\max_{\|u\|=1}
\left\{
u^{\mathrm T}Hu-\rho\mathcal V_X(u)
\right\}.
\]
Each Rayleigh quotient in this maximum is nonincreasing in $\rho$,
so $\phi_X$ is nonincreasing. At a simple dominant eigenvalue, its
derivative is given by the Hellmann--Feynman formula
\cite{Feynman1939,Lancaster1964,Magnus1985}, as recorded below.

\begin{proposition}
\label{prop:hf_sensitivity}
If the dominant eigenvalue of $M_\rho^{(X)}$ is simple at $\rho$, with
normalized eigenvector $u_\rho$, then
\begin{equation}
\label{eq:hf_derivative}
\phi_X'(\rho)
=
-u_\rho^{\mathrm T}(A^{\mathrm T}A+XX^{\mathrm T})u_\rho
=
-\mathcal V_X(u_\rho).
\end{equation}
\end{proposition}

For the positive penalty path in
\eqref{eq:penalty_characterization_sm}, the nonattainable case is
covered by Theorem~\ref{thm:main_convergence} applied to $-H$ and
$\widehat A_X$. Hence, as $\rho\to+\infty$,
\begin{equation}
\label{eq:max_first_order_penalty}
\phi_X(\rho)
=
\lambda_X^c+\frac{c_X}{\rho}+O(\rho^{-2}),
\quad c_X>0.
\end{equation}
At outer continuation level $j$, let $u_j$ be the normalized
approximate dominant eigenvector returned by the certified inner solve
at $\rho=\rho_j$, so that $u_j\approx u_{\rho_j}$.
Equation~\eqref{eq:hf_derivative} motivates
\[
g_j:=\mathcal V_X(u_j),\quad
\widehat c_j^{\rm HF}:=\rho_j^2g_j,\quad
\theta_j^{\rm HF}:=\theta_j^{\rm pen}-\rho_jg_j.
\]
For the exact simple branch, that is,
$u_j=u_{\rho_j}$ and
$\theta_j^{\rm pen}=\phi_X(\rho_j)$, a second-order perturbation
expansion in $1/\rho_j$, together with
\eqref{eq:hf_derivative}, gives
\[
\widehat c_j^{\rm HF}=c_X+O(\rho_j^{-1}),
\quad
\theta_j^{\rm HF}=\lambda_X^c+O(\rho_j^{-2}).
\]
The corrected value is used only as a numerical diagnostic and enters
neither the continuation update nor the stopping test.

Prediction is enabled only after two consecutive inner solves satisfy
$r_M\leq\varepsilon_{\rm in}$ and the HF data stabilize. Specifically,
\begin{equation}
\label{eq:hf_reliability}
r_{c,j}
=
\frac{|\widehat c_j^{\rm HF}-\widehat c_{j-1}^{\rm HF}|}
{\max\{c_{\min},\widehat c_j^{\rm HF},\widehat c_{j-1}^{\rm HF}\}},
\quad
r_{{\rm mon},j}
=
\frac{\max\{0,\theta_j^{\rm pen}-\theta_{j-1}^{\rm pen}\}}
{\max\{1,|\theta_{j-1}^{\rm pen}|\}}.
\end{equation}
The HF data are accepted when
$\widehat c_j^{\rm HF}>c_{\min}$,
$r_{c,j}\leq\tau_{\rm HF}$, and
$r_{{\rm mon},j}\leq\tau_{\rm mon}$.
The prediction tolerance $\varepsilon_{\rm pred}$ controls only the
estimated value tail and is not a stopping tolerance. The value-tail
target is
$\rho_{{\rm HF},j}^{\rm tar}
=\widehat c_j^{\rm HF}/
\{\varepsilon_{\rm pred}\max(1,|\theta_j^{\rm pen}|)\}$.
When the projection below is valid and
$r_{{\rm KKT},j}>\varepsilon_{\rm KKT}$, we also use the practical
first-order extrapolation
$\rho_{{\rm KKT},j}^{\rm tar}
=\rho_jr_{{\rm KKT},j}/\varepsilon_{\rm KKT}$.
This KKT target is used only after the HF reliability gate has passed.
With unavailable targets omitted,
\begin{equation}
\label{eq:rho_target}
\rho_j^{\rm tar}
=
\max\{\rho_{{\rm HF},j}^{\rm tar},\rho_{{\rm KKT},j}^{\rm tar}\},
\quad
\rho_{\rm next}
=
\min\{\rho_{\max},q_{\max}\rho_j,
      \max(q\rho_j,\rho_j^{\rm tar})\}.
\end{equation}
If the HF reliability test fails, the geometric update
$\rho_{\rm next}=\min\{\rho_{\max},q\rho_j\}$ is used. The cap
$q_{\max}$ in \eqref{eq:rho_target} limits the size of any predictor
jump.

\subsection{Projected certification}
\label{subsec:projected_certification}

Because $A$ has full row rank, $X^{\mathrm T}X=I_s$, and $AX=0$,
$A^{\mathrm T}(AA^{\mathrm T})^{-1}A$ and $XX^{\mathrm T}$ are the
orthogonal projectors onto $\mathcal R(A^{\mathrm T})$ and
$\operatorname{span}(X)$, respectively, and their ranges are
orthogonal. Thus the orthogonal projector onto $\mathcal S_X$ is
\begin{equation}
\label{eq:projector}
P_X
=
I-A^{\mathrm T}(AA^{\mathrm T})^{-1}A-XX^{\mathrm T},
\end{equation}
where the last term is omitted when $s=0$.
The projector need not be formed explicitly: to apply
$P_X$ to $v$, solve the system $AA^{\mathrm T}y=Av$ and set
$
P_Xv=v-A^{\mathrm T}y-X(X^{\mathrm T}v).
$

If $\|P_Xu_j\|>\tau_{\rm proj}$, define
$\widehat u_j=P_Xu_j/\|P_Xu_j\|$ and
$\widehat\theta_j=\widehat u_j^{\mathrm T}H\widehat u_j$.
The projected feasibility and KKT residuals are
\begin{equation}
\label{eq:projected_residuals}
r_{{\rm feas},j}
=
\bigl(\|A\widehat u_j\|^2+\|X^{\mathrm T}\widehat u_j\|^2\bigr)^{1/2},
\quad
r_{{\rm KKT},j}
=
\frac{\|P_X(H\widehat u_j-\widehat\theta_j\widehat u_j)\|}
{\max\{1,\|H\widehat u_j\|+|\widehat\theta_j|\}}.
\end{equation}
In exact arithmetic $r_{{\rm feas},j}=0$; it is retained to monitor
numerical errors in the projection step.

With $\widehat\theta_{\rm prev}$ denoting the preceding valid projected
Rayleigh quotient, set
\[
r_{{\rm out},j}
=
\frac{|\widehat\theta_j-\widehat\theta_{\rm prev}|}
{\max\{1,|\widehat\theta_j|,|\widehat\theta_{\rm prev}|\}}.
\]
The first valid projected level initializes
$\widehat\theta_{\rm prev}$. The certification can therefore occur only at
a subsequent valid projected level. Acceptance requires
\begin{equation}
\label{eq:outer_stop}
r_{M,j}\leq\varepsilon_{\rm in},\quad
r_{{\rm feas},j}\leq\varepsilon_{\rm feas},\quad
r_{{\rm KKT},j}\leq\varepsilon_{\rm KKT},\quad
r_{{\rm out},j}\leq\varepsilon_{\rm out}.
\end{equation}
Projection is used only for certification and reporting; continuation
follows the penalized dominant eigenbranch.
The projection step also remains matrix-free, since applying
$AA^{\mathrm T}$ requires only matrix-vector products with
$A^{\mathrm T}$ and $A$.

\subsection{Matrix-free Penalty--Split--Merge algorithm}
\label{subsec:psm_algorithm}

Algorithm~\ref{alg:penalty_sm} summarizes one constrained eigenpair
computation along the positive penalty path
\eqref{eq:penalty_characterization_sm}. Subsequent eigenpairs are
obtained by appending each accepted $\widehat u$ to $X$ and
reorthogonalizing the basis when needed.

\begin{algorithm}[!htbp]
\caption{Matrix-free Penalty--Split--Merge algorithm}
\label{alg:penalty_sm}
\begin{algorithmic}[1]
\REQUIRE Products with $H$, $A$, and $A^{\mathrm T}$; feasible
orthonormal basis $X$; bounds
$\underline\lambda_H,\overline\lambda_H,\overline\nu_A$;
$\rho_0>0$, $1<q\leq q_{\max}$,
$\rho_{\max}\geq\rho_0$, $j_{\max}\geq1$; tolerances
$\varepsilon_{\rm in}$, $\varepsilon_{\rm feas}$,
$\varepsilon_{\rm KKT}$, $\varepsilon_{\rm out}$,
$\varepsilon_{\rm pred}$, $\tau_{\rm proj}$; predictor parameters
$c_{\min}$, $\tau_{\rm HF}$, $\tau_{\rm mon}$; inner-safeguard
parameters described above.
\ENSURE Approximate constrained eigenpair
$(\widehat\theta,\widehat u)$, or \textsc{failure}.
\STATE Set $j=0$, $\rho_j=\rho_0$, choose $\|u^{(0)}\|=1$, and mark
$\widehat\theta_{\rm prev}$ unavailable.
\WHILE{$j<j_{\max}$}
\STATE Define $M_{\rho_j}^{(X)}$ and $\mathcal K_{\rho_j}^{(X)}$ and
apply safeguarded Split--Merge, warm-started with $u^{(j)}$, to obtain
$(\mu_j,u_j)$ with $\|u_j\|=1$.
\STATE Set
$\theta_j^{\rm pen}=\sigma_{\rho_j}\mu_j-\eta_{\rho_j}$ and compute
$r_{M,j}$; return \textsc{failure} if
$r_{M,j}>\varepsilon_{\rm in}$.
\STATE Compute
$g_j=\|Au_j\|^2+\|X^{\mathrm T}u_j\|^2$,
$\widehat c_j^{\rm HF}=\rho_j^2g_j$, and $p_j=P_Xu_j$.
\IF{$\|p_j\|>\tau_{\rm proj}$}
\STATE Form $(\widehat\theta_j,\widehat u_j)$ and compute
$r_{{\rm feas},j}$ and $r_{{\rm KKT},j}$.
\IF{$\widehat\theta_{\rm prev}$ is available}
\STATE Compute $r_{{\rm out},j}$; if \eqref{eq:outer_stop} holds,
\RETURN $(\widehat\theta_j,\widehat u_j)$.
\ENDIF
\STATE Set $\widehat\theta_{\rm prev}=\widehat\theta_j$.
\ENDIF
\IF{$\rho_j\geq\rho_{\max}$}
\RETURN \textsc{failure}.
\ENDIF
\STATE Set $\rho_{\rm next}=\min\{\rho_{\max},q\rho_j\}$.
\IF{$j\geq1$ and the HF reliability test
\eqref{eq:hf_reliability} holds}
\STATE Form $\rho_j^{\rm tar}$ from the available HF/KKT targets and
update $\rho_{\rm next}$ according to \eqref{eq:rho_target}.
\ENDIF
\STATE Set $u^{(j+1)}=u_j$, $\rho_{j+1}=\rho_{\rm next}$, and
$j\leftarrow j+1$.
\ENDWHILE
\RETURN \textsc{failure}.
\end{algorithmic}
\end{algorithm}

Each application of $M_\rho^{(X)}$ requires one product with $H$,
one with $A$, one with $A^{\mathrm T}$, and $O(ns)$ work for the thin
deflation products. The HF quantity reuses the same constraint products
when available. Krylov rescue and projected certification can also be
implemented using matrix-free products; hence neither
$A^{\mathrm T}A$, $XX^{\mathrm T}$, $P_X$, nor a null-space basis
needs to be formed explicitly.

\section{Numerical Experiments}
\label{sec:numerical_experiments}

The experiments examine the predicted penalty asymptotics, assess
the HF/KKT continuation and projected certification, compare PSM with
standard eigensolvers, and evaluate the safeguarded implementation on
larger-scale problems. We use \emph{PSM} to denote
Algorithm~\ref{alg:penalty_sm}. In the end-to-end runtime comparisons,
PSM uses the verified short-recurrence Lanczos safeguard with the
budget-80 trigger described below; \emph{PSM-LV} denotes this
implementation only in the large-scale comparison.

All experiments were performed in MATLAB R2017b on a Windows system
equipped with an Intel Core i7-7700HQ processor (2.8 GHz) and 24 GB of RAM.
The inner eigensolvers access $H$, $A$, and $A^{\mathrm T}$ only through
matrix-vector products. A K-app denotes one application of the
matrix-free penalized inner operator or its shifted/scaled form, and
$N_K$ denotes the total number of inner K-apps. Inner safeguard steps
and fresh-residual verification calls are included in $N_K$; outer
projection and certification work is included in the reported
wall-clock times.

\subsection{Experimental setup}
\label{subsec:numerical_setup}

Controlled matrix-free instances use coordinates in which
$\mathcal N(A)=\operatorname{span}\{e_1,\ldots,e_{n-m}\}$, with
$
A=[\,0,\Sigma\,],\quad
H=
\begin{bmatrix}
D&E^{\mathrm T}\\
E&B
\end{bmatrix},
$
where $D=\operatorname{diag}(d)$, $B$ and $\Sigma$ are diagonal,
$\Sigma$ is nonsingular, and
$E=0.10\,gh^{\mathrm T}$ with $\|g\|=\|h\|=1$.
We set $d_1=1$, $d_2=1-\Delta$, and $d_i<1-\Delta$ for $i\geq3$.
The target-coupling parameter is $|h_2|$, so
$\|Ee_2\|=0.10|h_2|$.

For deflated tests, $X=e_1$ and the target is $e_2$, hence
$\lambda_X^c=1-\Delta$ exactly. Unless stated otherwise,
$|h_2|=0$ in the attainable regime and $|h_2|=0.03$ in the weakly
nonattainable (WNA) regime. Fixed deterministic seed families are used,
and paired methods share the same instance and starting vector.
Except for the theory--algorithm bridge, which starts from a small
perturbation of the known target to identify the corresponding
penalized branch, reference eigenpairs are used only for post-processing
and are not used in stopping criteria. For these controlled instances,
the certification projector is applied directly in the known
coordinates; its cost is included in the reported wall-clock times.
Unless stated otherwise, the parameters are those in
Table~\ref{tab:algorithm_parameters}. The last two rows give the fixed
baseline safeguard parameters; the implementation used in the
end-to-end runtime comparisons is specified below.

\begin{table}[!htbp]
\centering
\small
\setlength{\tabcolsep}{3.5pt}
\caption{Default parameters used in the PSM experiments.}
\label{tab:algorithm_parameters}
\begin{tabular}{@{}lll@{}}
\toprule
Parameter & Default & Role\\
\midrule
$\rho_0,\ q,\ q_{\max}$
& $10^{-2},\,2,\,4$
& penalty continuation\\
$\rho_{\max},\ j_{\max}$
& $10^6,\,24$
& continuation safeguards\\
$\varepsilon_{\rm feas},\varepsilon_{\rm KKT},
  \varepsilon_{\rm out}$
& $10^{-6}$
& projected certification\\
$\varepsilon_{\rm in}$
& $5\times10^{-8}$
& penalized eigen-residual\\
$\varepsilon_{\rm pred},\tau_{\rm proj}$
& $\varepsilon_{\rm KKT},\,10^{-12}$
& prediction/projection\\
$c_{\min},\tau_{\rm HF},\tau_{\rm mon}$
& $10^{-12},\,0.05,\,10^{-7}$
& HF reliability test\\
$\varepsilon_\sigma,\delta_{\rm br}$
& $10^{-10},\,10^{-12}$
& inner/branch safeguards\\
check interval, $\tau_{\rm prog},m_{\rm stag}$
& $5,\,0.9,\,2$
& stagnation detection\\
rescue dimensions/escape cycles
& $12,24,36\,/\,96$
& adaptive Krylov safeguard\\
\bottomrule
\end{tabular}
\end{table}

\subsection{Theory validation and penalty calibration}
\label{subsec:theory_validation}

We use three complementary tests: an exactly solvable $2\times2$
model, a structured dense family with $n=192$, $m=40$,
$\Delta=10^{-3}$, and target-coupling parameter $0.25$ (five trials),
and a PSM theory--algorithm bridge with target couplings $0.05$, $0.15$,
and $0.30$ (three trials each).
Table~\ref{tab:theory_validation} summarizes the first-order behavior.

\begin{table}[!htbp]
\centering
\scriptsize
\setlength{\tabcolsep}{3.2pt}
\caption{First-order penalty asymptotics. The predicted log--log
slopes of the eigenvalue error $e_\rho$ and angle error
$\sin\angle(u_\rho,u_*)$ are $-1$, and $\rho e_\rho/c\to1$.
Structured-dense entries are five-trial medians; bracketed PSM entries
are ranges of the three coupling-wise medians.}
\label{tab:theory_validation}
\begin{tabular}{@{}lccc@{}}
\toprule
Test & Eigenvalue slope & Angle slope & Normalized constant\\
\midrule
Exact $2\times2$
& $-1.000007$
& $-1.000007$
& $0.999999$--$1.000001$\\
Structured dense
& $-0.999761$
& $-0.999750$
& $1.000013$\\
PSM bridge
& $[-0.9899,-0.9874]$
& $[-0.9893,-0.9867]$
& $[0.9942,0.9953]$\\
\bottomrule
\end{tabular}
\end{table}

The results are consistent with the predicted $O(\rho^{-1})$ error
and leading coefficient. All nine PSM bridge paths are certified; the
three coupling-wise median dense-reference eigenvalue slopes lie in
$[-0.9992,-0.9990]$, and the largest PSM inner residual is
$9.75\times10^{-9}$. In 20 exact-model calibration cases, the median
ratio of the penalty actually required to meet the target error to its
first-order prediction is $1.00365$.

For the continuation ablation at target tolerance $10^{-6}$,
geometric continuation and the HF-only value target both require
median counts of 14 levels and 2056 K-apps in the WNA regime. The
combined HF/KKT target with $q_{\max}=4$ reduces these medians to
13 levels and 1889 K-apps, with median final penalty $81.92$; the
median paired trialwise K-app reduction is $7.81\%$. Increasing
$q_{\max}$ to $8$ or $16$ does not reduce the median level count
further but raises the median final penalty to $146.31$. We therefore
fix $q_{\max}=4$ in the subsequent experiments.

\subsection{Accuracy, robustness, and end-to-end performance}
\label{subsec:accuracy_robustness}

To preserve the validation protocol, the operational-accuracy and
robustness tests use the fixed geometric grid
$\rho_j=10^{-2}2^j$ rather than the adaptive HF/KKT update. The latter
is assessed in the ablation above and used in the end-to-end
comparisons below. The accuracy test uses $n=2048$, $m/n=0.2$,
$\Delta=10^{-3}$, both regimes, five trials per case, and target
tolerances $\varepsilon=10^{-4},10^{-5},10^{-6}$. We set
$\varepsilon_{\rm feas}=\varepsilon_{\rm KKT}
=\varepsilon_{\rm out}=\varepsilon$ and
$\varepsilon_{\rm in}=\max\{10^{-9},\varepsilon/20\}$.

\begin{table}[!htbp]
\centering
\scriptsize
\setlength{\tabcolsep}{3.5pt}
\caption{Operational accuracy of PSM. Penalties are five-trial
medians; each residual entry is the larger of the attainable and WNA
medians.}
\label{tab:operational_accuracy}
\begin{tabular}{@{}cccccc@{}}
\toprule
Tol.
& $\rho_{\rm A}$
& $\rho_{\rm WNA}$
& max med. $r_{\rm KKT}$
& max med. $r_M$
& Certified\\
\midrule
$10^{-4}$
& $0.02$ & $0.64$
& $9.13\times10^{-5}$
& $4.26\times10^{-6}$
& $10/10$\\
$10^{-5}$
& $0.02$ & $10.24$
& $7.00\times10^{-6}$
& $1.80\times10^{-7}$
& $10/10$\\
$10^{-6}$
& $0.02$ & $81.92$
& $9.23\times10^{-7}$
& $3.65\times10^{-8}$
& $10/10$\\
\bottomrule
\end{tabular}
\end{table}

As shown in Table~\ref{tab:operational_accuracy}, all 30 runs
satisfy the certificate. The attainable cases certify
at the earliest eligible second level, $\rho=0.02$, whereas the WNA
penalty increases as the tolerance is tightened, consistently with the
finite-versus-asymptotic dichotomy. A broader test over
$n\in\{2048,8192\}$, $m/n\in\{0.1,0.2,0.3\}$,
$\Delta\in\{10^{-3},10^{-2}\}$, and both regimes gives 120/120
certified PSM runs; the largest groupwise median KKT and inner
residuals are $9.376\times10^{-7}$ and $4.586\times10^{-8}$.

We next compare PSM with baseline implementations of power
iteration~\cite{GolubVanLoan2013}, full-reorthogonalized
Lanczos~\cite{Lanczos1950,Parlett1998}, unpreconditioned
LOBPCG~\cite{Knyazev2001}, and the matrix-free, certification-matched
MATLAB \texttt{eigs} baseline, \texttt{eigs-cert}. The latter applies
the original penalized operator through a function handle, uses the
same initial inner tolerance, and accepts a Ritz vector only after a
fresh residual check; all retries and fresh-residual applications are
counted.
The 80 trials cover
$n\in\{1024,2048,4096,8192\}$, $m/n=0.2$,
$\Delta\in\{10^{-3},10^{-2}\}$, and both regimes. Five trials are
performed for each setting. All methods are evaluated on the same instances with the
same starting vectors, HF/KKT continuation rule, and certification
procedure; the execution order is balanced across methods.
PSM uses the verified short-recurrence Lanczos safeguard with the fixed
budget-$80$ trigger. Warm-up runs are excluded from the reported
wall-clock times.
For problem $p$, let $\mathcal C_p$ be the set of certified methods and
define
$r_{p,s}=T_{p,s}/\min\limits_{j\in\mathcal C_p}T_{p,j}$.
An uncertified run is assigned $r_{p,s}=\infty$; the performance
profile reports the fraction of all 80 problems with
$r_{p,s}\leq\tau$.
Figure~\ref{fig:performance_profile} and
Table~\ref{tab:paired_speedups} summarize the overall and paired
runtime comparisons, respectively.

\begin{figure}[!htbp]
\centering
\includegraphics[width=0.6\linewidth]
{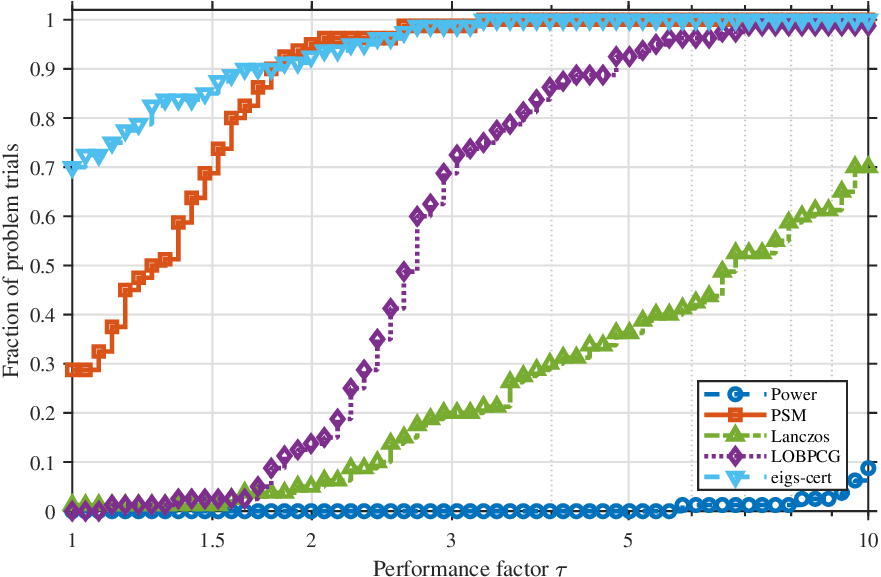}
\caption{End-to-end runtime performance profile for the moderate-scale
comparison. Uncertified runs receive infinite performance ratio, and
all 80 trials are included.}
\label{fig:performance_profile}
\end{figure}

\begin{table}[!htbp]
\centering
\scriptsize
\setlength{\tabcolsep}{3.2pt}
\caption{Paired end-to-end comparison with PSM. Ratios are medians of
the trialwise $T_{\rm comp}/T_{\rm PSM}$ and
$N_{K,\rm comp}/N_{K,\rm PSM}$ values over jointly certified pairs;
values above one favor PSM. The win fraction is the fraction of such
pairs with $T_{\rm PSM}<T_{\rm comp}$.}
\label{tab:paired_speedups}
\begin{tabular}{@{}lcccc@{}}
\toprule
Method
& Joint cert.
& $T_{\rm comp}/T_{\rm PSM}$
& PSM win frac.
& $N_{K,\rm comp}/N_{K,\rm PSM}$\\
\midrule
Power
& $38/80$
& $11.271$
& $1.000$
& $15.762$\\
Lanczos
& $80/80$
& $5.093$
& $0.975$
& $0.404$\\
LOBPCG
& $80/80$
& $1.932$
& $0.975$
& $1.741$\\
\texttt{eigs-cert}
& $80/80$
& $0.804$
& $0.300$
& $0.464$\\
\bottomrule
\end{tabular}
\end{table}

PSM, Lanczos, LOBPCG, and \texttt{eigs-cert} certify all 80 trials,
whereas power iteration certifies 38. PSM is fastest on $28.75\%$ of
the trials and is within a factor two of the best certified runtime on
$95\%$. It is faster than the tested Lanczos and LOBPCG implementations
in 78/80 pairs. Among the comparators, \texttt{eigs-cert} gives the
strongest wall-clock performance, with median
$T_{\rm eigs}/T_{\rm PSM}=0.804$. Thus PSM is competitive on this
benchmark, but the results do not support universal runtime dominance.
A fresh-seed WNA check at $n=2048$ and $8192$ also gives 13 penalty
levels, final penalty $81.92$, and full certification, with no
fresh-residual rejection. Relative to the same safeguarded
implementation without the proactive trigger, the budget-$80$ trigger
reduces the paired median K-app count by about $15.6\%$.

\subsection{Large-scale matrix-free behavior}
\label{subsec:large_scale}
The large-scale holdout uses a second seed family that was
not used to select the safeguard parameters. It consists of attainable
small-gap instances with
$n\in\{8192,16384,32768,65536,131072,262144\}$,
$m/n\approx0.08$, $\Delta=10^{-3}$, five paired trials per dimension,
and target tolerance $10^{-6}$.
Here \emph{PSM-LV} denotes PSM with the verified
short-recurrence Lanczos safeguard and the proactive budget-$80$
trigger, used at most once per penalized inner solve. The same
\texttt{eigs-cert} comparator is used with initial ARPACK tolerance
$\varepsilon_{\rm in}=5\times10^{-8}$ and fresh residual
verification. Table~\ref{tab:final_large_scale} summarizes the paired
results.

\begin{table}[!htbp]
\centering
\scriptsize
\setlength{\tabcolsep}{3.3pt}
\caption{Large-scale holdout. The PSM-LV and \texttt{eigs-cert}
columns give median wall-clock times in seconds. Ratio columns are
medians of the five trialwise PSM-LV/\texttt{eigs-cert} ratios; values
below one favor PSM-LV. ``LV wins'' counts paired trials with smaller
PSM-LV runtime.}
\label{tab:final_large_scale}
\begin{tabular}{@{}rrrrrr@{}}
\toprule
$n$ & PSM-LV & \texttt{eigs-cert}
& $T_{\rm LV}/T_{\rm eigs}$ & LV wins
& $N_{K,\rm LV}/N_{K,\rm eigs}$\\
\midrule
8192   & 0.041 & 0.046 & 0.849 & 5/5 & 1.202\\
16384  & 0.073 & 0.086 & 0.810 & 4/5 & 1.096\\
32768  & 0.111 & 0.217 & 0.475 & 5/5 & 1.096\\
65536  & 0.245 & 0.365 & 0.671 & 5/5 & 1.097\\
131072 & 0.888 & 0.960 & 0.925 & 5/5 & 1.097\\
262144 & 2.225 & 2.332 & 0.956 & 5/5 & 1.097\\
\bottomrule
\end{tabular}
\end{table}

As shown in Table~\ref{tab:final_large_scale}, all 30 pairs are
certified, and PSM-LV is faster in 29/30 paired trials. At the two
largest dimensions it is faster in all five trials, with median paired
time ratios $0.925$ and $0.956$. For $n\geq16384$, the paired $N_K$
ratio is about $1.097$, so the observed runtime advantage is not
explained by fewer operator applications. Every penalized
\texttt{eigs-cert} solve passes the fresh residual check on the first
ARPACK call, with no tolerance-tightening retry.
For these runs, the basis-storage model counts 13 basis/work vectors
for PSM-LV and 20 for ARPACK, corresponding to 26 and 40 MiB,
respectively, at $n=262144$. Common problem storage and MATLAB overhead
are excluded; these are storage-model estimates rather than measured
peak process memory.

\section{Conclusions}
\label{sec:conclusion}

This paper developed a quantitative perspective on penalty paths for
constrained extremal eigenvalue problems, extending classical
qualitative characterizations to quantitative spectral relations. By
relating the reduced-space eigenvalues to the eigenvalues of full-space
penalized problems, the analysis characterized the penalty path
behavior and identified the conditions for finite recovery. When finite
recovery did not occur, the first-order asymptotic expansion quantified
the decay of the penalty error and provided information for penalty
parameter prediction. These results also led to a matrix-free
Penalty--Split--Merge framework that combined penalty continuation,
Split--Merge inner iterations, and projected certification. The
framework avoided explicit null-space bases and penalty matrices while
accessing the relevant operators through matrix-vector products.
Numerical experiments examined the penalty behavior, certification
procedure, and computational performance of the proposed framework on
problems of different scales.
Future work includes extending the framework to constrained generalized eigenvalue problems, developing adaptive strategies for automatic selection of the penalty parameter schedule, and integrating the proposed certification mechanism into black-box eigensolvers for large-scale applications.

\section*{Acknowledgments}

The authors used generative AI tools to assist with language editing and
presentation, the organization of related literature, and the refinement
of algorithm descriptions and numerical experiment code. All mathematical
statements, proofs, algorithms, references, code, and numerical results
were reviewed and verified by the authors. The authors assume
responsibility for all content.

\bibliographystyle{siamplain}
\bibliography{references-new}

@article{AnstreicherWright00,
  author  = {Anstreicher, Kurt M. and Wright, Margaret H.},
  title   = {A note on the augmented Hessian when the reduced Hessian is semidefinite},
  journal = {SIAM J. Optim.},
  volume  = {11},
  pages   = {243--253},
  year    = {2000}
}

@article{Benzi2005,
  author  = {Benzi, Michele},
  title   = {Preconditioning techniques for large linear systems: A survey},
  journal = {J. Comput. Phys.},
  volume  = {182},
  pages   = {418--477},
  year    = {2002}
}

@book{BoydElGhaouiFeronBalakrishnan94,
  author    = {Boyd, Stephen and {El Ghaoui}, Laurent and Feron, Eric and Balakrishnan, Venkataramanan},
  title     = {Linear Matrix Inequalities in System and Control Theory},
  series    = {SIAM Studies in Applied Mathematics},
  volume    = {15},
  publisher = {SIAM},
  address   = {Philadelphia, PA},
  year      = {1994}
}

@article{Feynman1939,
  author  = {Feynman, Richard P.},
  title   = {Forces in molecules},
  journal = {Phys. Rev.},
  volume  = {56},
  pages   = {340--343},
  year    = {1939}
}

@article{Finsler37,
  author  = {Finsler, Paul},
  title   = {{\"U}ber das Vorkommen definiter und semidefiniter Formen in Scharen quadratischer Formen},
  journal = {Comment. Math. Helv.},
  volume  = {9},
  pages   = {188--192},
  year    = {1937}
}

@book{GolubVanLoan2013,
  author    = {Golub, Gene H. and Van Loan, Charles F.},
  title     = {Matrix Computations},
  edition   = {4th},
  publisher = {Johns Hopkins University Press},
  address   = {Baltimore, MD},
  year      = {2013}
}

@article{GouldHribarNocedal01,
  author  = {Gould, Nicholas I. M. and Hribar, Mary E. and Nocedal, Jorge},
  title   = {On the solution of equality constrained quadratic programming problems arising in optimization},
  journal = {SIAM J. Sci. Comput.},
  volume  = {23},
  pages   = {1376--1395},
  year    = {2001}
}

@article{Hestenes69,
  author  = {Hestenes, Magnus R.},
  title   = {Multiplier and gradient methods},
  journal = {J. Optim. Theory Appl.},
  volume  = {4},
  pages   = {303--320},
  year    = {1969}
}

@book{HornJohnson12,
  author    = {Horn, Roger A. and Johnson, Charles R.},
  title     = {Matrix Analysis},
  edition   = {2nd},
  publisher = {Cambridge University Press},
  address   = {Cambridge},
  year      = {2012}
}

@article{Knyazev2001,
  author  = {Knyazev, Andrew V.},
  title   = {Toward the optimal preconditioned eigensolver: Locally optimal block preconditioned conjugate gradient method},
  journal = {SIAM J. Sci. Comput.},
  volume  = {23},
  pages   = {517--541},
  year    = {2001}
}

@article{Lancaster1964,
  author  = {Lancaster, Peter},
  title   = {On eigenvalues of matrices dependent on a parameter},
  journal = {Numer. Math.},
  volume  = {6},
  pages   = {377--387},
  year    = {1964}
}

@article{Lanczos1950,
  author  = {Lanczos, Cornelius},
  title   = {An iteration method for the solution of the eigenvalue problem of linear differential and integral operators},
  journal = {J. Res. Natl. Bur. Stand.},
  volume  = {45},
  pages   = {255--282},
  year    = {1950}
}

@article{LiuSongXia2026,
  author  = {Liu, Xiaozhi and Song, Mengmeng and Xia, Yong},
  title   = {{Split--Merge}: A difference-based approach for dominant eigenvalue problem},
  journal = {SIAM J. Optim.},
  year    = {2026},
  note    = {to appear}
}

@article{Magnus1985,
  author  = {Magnus, Jan R.},
  title   = {On differentiating eigenvalues and eigenvectors},
  journal = {Econometric Theory},
  volume  = {1},
  pages   = {179--191},
  year    = {1985}
}

@book{NocedalWright06,
  author    = {Nocedal, Jorge and Wright, Stephen J.},
  title     = {Numerical Optimization},
  edition   = {2nd},
  publisher = {Springer},
  address   = {New York},
  year      = {2006}
}

@book{Parlett1998,
  author    = {Parlett, Beresford N.},
  title     = {The Symmetric Eigenvalue Problem},
  publisher = {SIAM},
  address   = {Philadelphia, PA},
  year      = {1998}
}

@article{PolikTerlaky07,
  author  = {P{\'o}lik, Imre and Terlaky, Tam{\'a}s},
  title   = {A survey of the {S}-lemma},
  journal = {SIAM Rev.},
  volume  = {49},
  pages   = {371--418},
  year    = {2007}
}

@incollection{Powell69,
  author    = {Powell, Michael J. D.},
  title     = {A method for nonlinear constraints in minimization problems},
  booktitle = {Optimization},
  editor    = {Fletcher, R.},
  publisher = {Academic Press},
  address   = {London},
  pages     = {283--298},
  year      = {1969}
}

@book{Saad11,
  author    = {Saad, Yousef},
  title     = {Numerical Methods for Large Eigenvalue Problems},
  edition   = {revised},
  publisher = {SIAM},
  address   = {Philadelphia, PA},
  year      = {2011}
}

@article{ShaoChen25,
  author  = {Shao, Nian and Chen, Wenbin},
  title   = {Riemannian acceleration with preconditioning for symmetric eigenvalue problems},
  journal = {Numer. Math.},
  volume  = {157},
  pages   = {307--354},
  year    = {2025}
}

@article{ShaoChenBai25,
  author  = {Shao, Nian and Chen, Wenbin and Bai, Zhaojun},
  title   = {{EPIC}: A provable accelerated eigensolver based on preconditioning and implicit convexity},
  journal = {SIAM J. Matrix Anal. Appl.},
  volume  = {46},
  pages   = {45--73},
  year    = {2025}
}

@book{Strang2016,
  author    = {Strang, Gilbert},
  title     = {Introduction to Linear Algebra},
  edition   = {5th},
  publisher = {Wellesley-Cambridge Press},
  address   = {Wellesley, MA},
  year      = {2016}
}

@article{VandenbergheBoyd96,
  author  = {Vandenberghe, Lieven and Boyd, Stephen},
  title   = {Semidefinite programming},
  journal = {SIAM Rev.},
  volume  = {38},
  pages   = {49--95},
  year    = {1996}
}

@article{ZhouBaiLi21,
  author  = {Zhou, Yunshen and Bai, Zhaojun and Li, Ren-Cang},
  title   = {Linear constrained {Rayleigh} quotient optimization: Theory and algorithms},
  journal = {CSIAM Trans. Appl. Math.},
  volume  = {2},
  pages   = {195--262},
  year    = {2021}
}

\end{document}